\documentclass{article}
\usepackage{amsfonts,amsmath,amsthm,amssymb}
\usepackage[math]{kurier}
\numberwithin{equation}{section}
\newtheorem{ssmptn}{Assumption}
\DeclareMathOperator*{\dive}{div}
\DeclareMathOperator*{\esssup}{ess\,sup}

\newcommand{\al}{\alpha}
\newcommand{\bino}{\bigskip\noindent}
\newcommand{\lmi}{\mathop{\mbox{\rm liminf}}\limits}

\theoremstyle{plain}
\newtheorem{thrm}{Theorem}[section]
\newtheorem{lmm}[thrm]{Lemma}

\newtheorem{prpstn}[thrm]{Proposition}

\theoremstyle{definition}
\newtheorem{dfntn}[thrm]{Definition}

\newtheorem{rmrk}[thrm]{Remark}

\title{Stability in affine optimal control problems constrained by semilinear elliptic partial differential equations\thanks{This research was supported by 
		the Austrian Science Foundation (FWF) under grant I 4571-N.}}

\author{
	Alberto Dom\'inguez Corella\thanks{Institute of Statistics and Mathematical Methods in Economics,
		Vienna University of Technology, Austria, {\tt alberto.corella@tuwien.ac.at}} 
	\and  
	Nicolai Jork \thanks{Institute of Statistics and Mathematical Methods in Economics,
		Vienna University of Technology, Austria, {\tt nicolai.jork@tuwien.ac.at}} 
	\and  
	Vladimir Veliov\thanks{Institute of Statistics and Mathematical Methods in Economics,
		Vienna University of Technology, Austria, {\tt vladimir.veliov@tuwien.ac.at}}}

\date{}

\begin{document}
\maketitle

\begin{abstract} 
	This paper investigates stability properties of affine optimal control problems constrained 
	by semilinear elliptic partial differential equations. This is done by studying the so called metric subregularity 
	of the set-valued mapping associated with the system of first order necessary optimality conditions. 
	Preliminary results concerning the differentiability of the functions involved are established, 
	especially the so-called switching function. Using this ansatz, more general nonlinear perturbations are encompassed, 
	and under weaker assumptions than the ones previously considered in the literature on 
	control constrained elliptic problems. 
	Finally, the applicability of the results is illustrated with some error estimates for the Tikhonov regularization.
\end{abstract}
\section{Introduction}
We consider the following  optimal control problem
\begin{align}\label{cost}
\min_{u\in\mathcal U}\left\lbrace\int_{\Omega}\Big[w(x,y)+  s(x,y) u\Big]\,dx\right\rbrace ,
\end{align}
subject to 
\begin{align}\label{system}
\left\{ \begin{array}{cclcc}
-\dive\big(A(x)\nabla y\big)+d(x,y)&=&\beta(x)u& \text{in}& \Omega  \\
\\ A(x)\nabla y\cdot \nu+b(x)y&=&0 &\text{on}& \partial\Omega.
\end{array} \right.
\end{align}
The set $\Omega\subset\mathbb R^n$ is a bounded domain with Lipschitz boundary, where 
$n\in\left\lbrace 2,3\right\rbrace $. The unit outward normal vector field on the boundary $\partial\Omega$, which is single valued a.e. in $\partial\Omega$, is denoted by $\nu$.  The control set is given by
\begin{align*}
\mathcal U:=\left\lbrace u:\Omega\to\mathbb R\hspace*{0.2cm}\text{measurable}:  \hspace{0.20cm}b_1(x)\le u(x)\le b_2(x) \hspace{0.20cm}\text{for a.e. $x\in\Omega$}\right\rbrace, 
\end{align*}
where $b_{1}$ and $b_{2}$ are bounded measurable 
functions satisfying $b_1(x)\le b_2(x)$ for a.e. $x\in\Omega$. The functions $w: \Omega\times\mathbb R \to \mathbb R$, 
$s:\Omega\times\mathbb R\to\mathbb R$,  $d:\Omega\times\mathbb R\to\mathbb R$,  $\beta : \Omega \to \mathbb R$ and 
$b:\partial\Omega\to\mathbb R$
are real-valued and measurable, and $A:\Omega\to \mathbb R^{n\times n}$ is a measurable matrix-valued function.

There are many motivations for studying stability of solutions, in particular for error analysis of numerical methods, 
see e.g., \cite{Pornerelliptic2,Pornerelliptic1}.   
Most of the stability results for elliptic control problems are obtained under a second order growth condition 
(analogous to the classical Legendre-Clebsch condition). For literature concerning this type of problems,  the reader  is referred
to \cite{Grielliptic,Hinzeelliptic,Kiebelliptic,Malaelliptic,Morduelliptic,Quielliptic} and the references therein. 
In optimal control problems like (\ref{cost})--(\ref{system}), where the control appears linearly
(hence, called affine problems) this growth condition does not hold. The so-called 
bang-bang solutions are ubiquitous in this case, see \cite{Casas93,Casasbang,Casanum}. 
To give an account of the state of art in stability of bang-bang problems, we  mention the works
\cite{Sey2,SubregOsm,Mayersubreg,Prei,Sey1} on optimal control of ordinary differential equations. 
Associated results for optimization problems constrained by partial differential equations have been gaining 
relevance in recent years, see \cite{Casascone,CT,Casasbang,Casanum,Hinzestruct,Wachelliptic}. 
However,  its stability has been only investigated in a handful of papers, see e.g.,
\cite{Hinzestruct,MR3810878,Wachelliptic}. 
From these works, we mention here particularly \cite{Wachelliptic}, 
where the authors consider linear perturbations in the state and adjoint equations for a similar problem with 
Dirichlet boundary condition.  They use the so-called structural assumption (a growth assumption satisfied 
near the jumps of the control) on the adjoint variable. This assumption has been widely used in 
the literature on bang-bang control of ordinary differential equations in a somewhat different form, see, e.g., \cite{Sey2,SubregOsm,Prei,Sey1}.

The investigations of stability properties of optimization problems, in general, are usually based on 
the study of similar properties of the corresponding system of necessary optimality conditions. 
The first order necessary optimality conditions for problem (\ref{cost})--(\ref{system}) can be 
recast as a system of two elliptic equations (primal and adjoint) and one variational inequality
(representing the minimization condition of the associated Hamiltonian), forming together a {\em generalized
	equation}, that is, an inclusion involving a set-valued mapping called {\em optimality mapping}. 
The concept of {\em strong metric subregularity}, see \cite{Subreg,Dontchevbook},
of set-valued mappings has shown to be efficient in many applications especially ones related 
to error analysis, see \cite{Bonnans}.
This also applies to optimal control problems of ordinary differentials equations, see e.g., \cite{MpcDont,SubregOsm}.

In the present paper we investigate the strong metric subregularity property of the optimality mapping
associated with problem (\ref{cost})--(\ref{system}). We present sufficient conditions for strong subregularity 
of this mapping on weaker assumptions than the ones used in literature, see Section \ref{Section5} for precise details.
The structural assumption in \cite{Wachelliptic} is weakened and more general perturbations are considered. 
Namely, perturbations in the variational inequality, appearing as a part of the first 
order necessary optimality conditions, are considered; which are important in the numerical analysis of ODE and 
PDE constrained optimization problems. Moreover,  nonlinear perturbations are investigated, 
which provides a framework for applications, as illustrated with an estimate related to the Tikhonov regularization.
The concept of linearization is employed in a functional frame in order to deal with nonlinearities. 
The needed differentiability of the control-to-adjoint mapping and the switching function (see Section \ref{Section diff})
is proved, and the derivatives are used to obtain adequate estimates needed in the stability results. 
Finally, we consider nonlinear perturbations in a general framework. We propose the use of the compact-open 
topology to have a notion of “closeness to zero" of the perturbations. In our particular case this topology can me 
metrized, providing a more “quantitative" notion. Estimates in this metric are obtained in 
Section \ref{Section nonlin}.

\section{Preliminaries}
The  euclidean space $\mathbb R^s$ is considered with its usual norm, denoted by $|\cdot|$. 
As usual, for $p\in[1,\infty)$, we denote by $L^p(\Omega)$ the space of all measurable $p$-integrable functions 
$\psi:\Omega\to\mathbb R^s$ with the norm
\begin{align*}
|\psi|_{L^p(\Omega)}:=\Big(\sum_{i=1}^{s}\int_{\Omega}|\psi_i(x)|^p\,dx\Big)^{\frac{1}{p}}.
\end{align*}
The space $L^\infty(\Omega)$ consists of all measurable essentially bounded functions $\psi:\Omega\to\mathbb R^s$ with the norm
\begin{align*}
|\psi|_{L^\infty(\Omega)}:=\esssup_{x\in\Omega}|\psi(x)|.
\end{align*}
We denote by $C(\bar\Omega)$ the space of continuous functions on $\Omega$ that can be extended 
continuously to $\bar\Omega$ equipped with the $L^\infty$-norm.  We denote  by $H^1(\Omega)$ the space 
of functions $\psi\in L^2(\Omega)$  having all first order weak derivatives in $L^2(\Omega)$ endowed with its usual norm.
The space $H^1(\Omega)\cap C(\bar \Omega)$ is endowed with the norm
\begin{align*}
|\psi|_{H^1(\Omega)\cap C(\bar \Omega)}:= |\psi|_{H^1(\Omega)}+|\psi|_{C(\bar \Omega)}.
\end{align*}
A function $\psi:\Omega\times\mathbb R\to\mathbb R$ is said to be Carath\'eodory if $\psi(\cdot,y)$ is measurable 
for every $y\in\mathbb R$, and $\psi(x,\cdot)$ is continuous for a.e. $x\in\Omega$.
A function $\psi:\Omega\times\mathbb R\to\mathbb R$ is said to be  locally Lipschitz, uniformly in the first variable, if 
for each $M>0$ there exists $L>0$ such that 
\begin{align*}
|\psi(x,y_2)-\psi(x,y_1)|\le L|y_2-y_1|
\end{align*} 
for a.e. $x\in\Omega$ and all $y_1,y_2\in [-M,M]$.
In order to abbreviate notation, we define $f,g:\Omega\times\mathbb R\times\mathbb R\to\mathbb R$ by
\begin{align*}
f(x,y,u) := \beta(x) u-d(x,y) \quad \text{and}     \quad   g(x,y,u) := w(x,y) +  s(x,y) u.
\end{align*}
The following assumption is supposed to hold throughout the remainder of the paper. It ensures that 
the mathematical  objects related to problem (\ref{cost})--(\ref{system})  that we consider are well defined. 
Assumption \ref{A1} is quite standard in the literature, see  the book \cite{TroeltzschPde}.
\begin{ssmptn}\label{A1}
	The following statements are assumed to hold.
	\begin{itemize}
		\item[(i)] The set $\Omega\subset\mathbb R^n$ is a bounded Lipschitz domain. 
		The matrix $A(x)$ is symmetric for a.e. $x$ in $\Omega$, and there exists $\alpha>0$ such that 
		$\xi \cdot A(x)\xi\ge \alpha|\xi|^2$ for a.e. $x$ in $\Omega$ and all $\xi\in\mathbb R^n$.
		
		\item[(ii)] The functions $w,s$ and $d$ are Carath\'eodory, twice differentiable with respect to 
		the {second} variable, and their second derivatives are locally Lipschitz, uniformly in the first variable. 
		
		\item[(iii)] The functions $A,\beta,b,d(\cdot,0),d_{y}(\cdot,0),w_y(\cdot,0)$ and $s_y(\cdot,0)$ are measurable and bounded.
		
		\item[(iv)] The function $d_y(\cdot,y)$ is nonnegative a.e. in $\Omega$ for all $y\in\mathbb R$. 
		The function $b$ is nonnegative a.e. in $\partial\Omega$ and $|b|_{L^\infty(\partial\Omega)}>0$.
	\end{itemize}
\end{ssmptn}

Items $(i)$ and $(iv)$ of Assumption \ref{A1} ensure that the partial differential equations appearing in this paper have 
unique solutions in the space $H^1(\Omega)\cap L^\infty(\Omega)$.

\subsection{The elliptic operator}
We consider the set $D(\mathcal L)$ of all functions $y\in H^1(\Omega)\cap L^\infty(\Omega)$ for which there exists 
$h\in L^2(\Omega)$ such that 
\begin{align}\label{weakformulation}
\int_{\Omega}A(x)\nabla y\cdot\nabla\varphi\,dx+\int_{\partial\Omega}b(x)y\varphi\,ds(x)=\int_{\Omega}h\varphi\,dx
\quad\forall \varphi\in H^1(\Omega).
\end{align}
As usual, $ds$ denotes the Lebesgue surface measure. It is easy to see that for each $y\in D(\mathcal L)$ there 
exists a unique element $h\in L^2(\Omega)$ such that (\ref{weakformulation}) holds. We define the operator 
$\mathcal L:D(\mathcal L)\to L^2(\Omega)$ by assigning each $y\in D(\mathcal L)$ to the function $h\in L^2(\Omega)$ 
satisfying (\ref{weakformulation}). 
By definition, a function $y\in H^1(\Omega)\cap L^\infty(\Omega)$ belongs to $D(\mathcal L)$ if, and only if, it is the weak solution of the linear elliptic partial differential equation
\begin{align*}
\left\{ \begin{array}{cclcc}
-\dive\big(A(x)\nabla y\big)&=&h& \text{in}& \Omega, \\
\\ A(x)\nabla y\cdot \nu+b(x)y&=&0&\text{on}& \partial\Omega
\end{array} \right.
\end{align*}
for some $h\in L^2(\Omega)$. The following lemma is of trivial nature.
\begin{lmm}
	The set $D(\mathcal L)$ is a linear subspace of $H^1(\Omega)\cap L^\infty(\Omega)$. Moreover, 
	the operator $\mathcal L:D(\mathcal L)\to L^2(\Omega)$ is a well defined linear mapping.
\end{lmm}

If $D(\mathcal L)$ is endowed with the norm of $L^2(\Omega)$, then $\mathcal L$ is an unbounded operator 
from $D(\mathcal L)$ to $L^2(\Omega)$. Since $A(x)$ is symmetric for a.e. $x\in\Omega$, by (\ref{weakformulation}) 
we have 
\begin{align}\label{intbyparts}
\int_{\Omega} \mathcal Ly\varphi\, dx=\int_\Omega y\mathcal L \varphi\, dx
\end{align}
for all $y,\varphi\in D(\mathcal L)$, the so-called integration by parts formula.

\begin{rmrk}
	If $\partial\Omega$ is of class $C^{1,1}$, $A$ is Lipschitz in $\bar\Omega$, and $b$ is Lipschitz and 
	positive in $\partial\Omega$, then
	\begin{align*}
	D(\mathcal L)= \{y\in H^2(\Omega):A(\cdot)\nabla y\cdot \nu+b(\cdot)y=0 \},
	\end{align*}
	and $\mathcal Ly=-\dive\big(A(\cdot)\nabla y\big)$ for all $y\in  D(\mathcal L)$, see \cite[Theorem 2.4.2.6]{Grisvard}.
\end{rmrk}
The following lemma shows the inclusion $D(\mathcal L)\subset C(\bar\Omega)$. Its proof can be found in 
\cite[Theorem 4.7]{TroeltzschPde} and follows the arguments in \cite{Casas93,Stampacchiaelliptiques}.

\begin{lmm}\label{L1e}
	Let $\alpha\in L^\infty(\Omega)$ be nonnegative and $h\in L^2(\Omega)$. There exists a unique 
	function $y\in D(\mathcal L)$ such that 
	\begin{align}\label{linequ}
	\mathcal Ly+\alpha(\cdot)y=h
	\end{align}
	and this function belongs to $C(\bar\Omega)$. Moreover, for each $r>n/2$ there exists a positive number $c$ such that 
	\begin{align*}
	|y|_{H^1(\Omega)\cap C(\bar\Omega)}\le c|h|_{L^r(\Omega)}
	\end{align*}
	for all $\alpha\in L^\infty(\Omega)$ nonnegative, $y\in D(\mathcal L)$, and $h\in L^2(\Omega)\cap L^r(\Omega)$ 
	satisfying (\ref{linequ}).
\end{lmm}

The following technical lemma can be deduced from Lemma \ref{L1e}, see the proof of \cite[Lemma 3.4]{Casanum}. Its use in optimal control of elliptic
partial differential equations dates from the paper \cite[Lemma 2.6]{Casasbang}. It has shown to be useful 
for diverse estimates, see \cite{Casasbang,Wachelliptic}.

\begin{lmm}\label{L2e}
	There exists a positive number $c$ such that 
	\begin{align*}
	|y|_{L^2(\Omega)}\le c|h|_{L^1(\Omega)}
	\end{align*}
	for all nonnegative $\alpha\in L^\infty(\Omega)$, $y\in D(\mathcal L)$ and $h\in L^2(\Omega)$ satisfying (\ref{linequ}).
\end{lmm}

{The proof of the next result can be found in
	\cite[Theorem 2.11]{Casasnonmonotone} in the case of a Dirichlet problem, see also \cite[Lemma 6.8]{Dintelmann}. Here we adapt the  
	argument below Theorem 2.1 in \cite[p. 618]{Delosreyes}}.

\begin{lmm}\label{weakconv}
	Let $\alpha\in L^\infty(\Omega)$ be nonnegative,  $\{h_m\}_{m=1}^\infty$ be a sequence in $L^2(\Omega)$ 
	and $h\in L^2(\Omega)$. For each $m\in\mathbb N$, let $y_m\in C(\bar\Omega)$ be the unique function 
	satisfying $\mathcal L y_m+\alpha(\cdot)y_m=h_m$, and let $y\in C(\bar\Omega)$ be the unique function satisfying 
	of $\mathcal Ly+\alpha(\cdot)y=h$. If $h_m\rightharpoonup h$ weakly in $L^2(\Omega)$, then $y_m\to y$ in $C(\bar\Omega)$.
\end{lmm}

\begin{proof}
	Let  $p\in(2n/(n+2),n/(n-1))$.
	Then  $W^{1,p}(\Omega)$  is compactly embedded in $L^2(\Omega)$ and consequently, by Schauder's  Theorem, 
	$L^2(\Omega)$ is compactly embedded in $W^{1,p}(\Omega)^*$. By the latter compact embedding, every weakly 
	convergent sequence in $L^2(\Omega)$ converges also in $W^{1,p}(\Omega)^*$ to the same limit.  
	Define $\mathcal K:L^2(\Omega)\to C(\bar\Omega)$ by $\mathcal Kh:= y$, where $y\in C(\bar\Omega)$ is the unique 
	function satisfying $\mathcal Ly+\alpha(\cdot)=h$.
	The result follows from \cite[Theorem 3.14]{Robin}, since that theorem asserts that the linear operator $\mathcal K$ 
	is continuous from $L^2(\Omega)$ endowed with the norm of $W^{1,p}(\Omega)^*$ to $C(\bar\Omega)$.
\end{proof}

\begin{rmrk}\label{remark1}
	Using the definitions of the set $D(\mathcal L)$ and the operator $\mathcal L$, we can write in a shorter way 
	the partial differential equations involved in this paper.  For example, given $u\in\mathcal U$, to say that $y$ 
	belongs to $D(\mathcal L)$ and satisfies  $\mathcal Ly+d(\cdot,y)=\beta(\cdot)u$ 
	is equivalent to say that $y$ belongs to $ H^1(\Omega)\cap L^\infty(\Omega)$ and satisfies the weak formulation of (\ref{system}), that is
	\begin{align*}
	\int_{\Omega}A(x)\nabla y\cdot\nabla\varphi\,dx+\int_{\Omega}d(x,y)\varphi\,dx+
	\int_{\partial\Omega}b(x)y\varphi\,ds(x)=\int_{\Omega}\beta(x)u \varphi\,dx
	\end{align*}
	for all $\varphi\in H^1(\Omega)$. This weak formulation makes sense since, by $(ii)$ and $(iii)$ of Assumption \ref{A1}, 
	for any $y\in L^\infty(\Omega)$, the function $d(\cdot,y)$ belongs to $L^\infty(\Omega)$.
\end{rmrk}


\subsection{The control model}

Having in mind Remark \ref{remark1}, given a function $u\in\mathcal U$ we say that $y_u\in D(\mathcal L)$ 
is the associated state to $u\in\mathcal U$ if
\begin{align}\label{stateeq}
\mathcal Ly_u=f(\cdot,y_u,u).
\end{align}

The following proposition shows that the {mapping} $u\to y_u$ from $ \mathcal{U}$ to $D(\mathcal L)$ 
is well defined. Its proof can be found in the standard literature; it follows from 
\cite[Theorem 4.8]{TroeltzschPde}, see also \cite[p. 212]{TroeltzschPde}.

\begin{prpstn}\label{contstate}
	For each $u\in\mathcal U$ there exists a unique state $y_u\in D(\mathcal L)$ associated with $u\in\mathcal U$. 
	Moreover,  $\{y_u:u\in\mathcal U\}$ is a bounded subset of $H^1(\Omega)\cap C(\bar\Omega)$ and for each $r>n/2$ 
	there exists $c>0$ such that
	\begin{align*}
	|y_{u_2}-y_{u_1}|_{H^1(\Omega)\cap C(\bar\Omega)}\le c|u_2-u_1|_{L^r(\Omega)}
	\end{align*}
	for all $u_1,u_2\in\mathcal U$.
\end{prpstn}

We call the function $\mathcal G:\mathcal U\to H^1(\Omega)\cap C(\bar\Omega)$ given by $\mathcal G(u):= y_u$ 
the control-to-state mapping. 
The functional $\mathcal J:\mathcal U\to\mathbb R$ given by
\begin{align*}
\mathcal J(u):=\int_{\Omega}g(x,y_{u},u)\,dx
\end{align*}
is called the objective functional of problem (\ref{cost})--(\ref{system}).

\begin{dfntn}
	Let $\bar u$ belong to $\mathcal U$.
	\begin{itemize}
		\item[(i)] We say that $\bar u$ is a global solution of problem (\ref{cost})--(\ref{system}) if  
		$\mathcal J(\bar u)\le\mathcal J(u)$  for all $u\in\mathcal U$.
		
		\item[(ii)] We say that $\bar u$ is a local solution of problem (\ref{cost})--(\ref{system}) if there exists 
		$\varepsilon_0>0$ such that $\mathcal J(\bar u)\le\mathcal J(u)$  for all $u\in\mathcal U$ 
		with $|u-\bar u|_{L^1(\Omega)}\le\varepsilon_0$.
		
		\item[(iii)] We say that $\bar u$ is a strict local solution of problem (\ref{cost})--(\ref{system}) if 
		there exists $\varepsilon_0>0$ such that $\mathcal J(\bar u)<\mathcal J(u)$  for all $u\in\mathcal 
		U$ with $u\neq \bar u$  and $|u-\bar u|_{L^1(\Omega)}\le\varepsilon_0$.
	\end{itemize}
\end{dfntn}

Under Assumption \ref{A1}, problem (\ref{cost})--(\ref{system}) has at least one global solution. 
The proof is routine and can be obtained by standard arguments; namely, taking a minimizing sequence 
and using the weak compactness of $\mathcal U$ in $L^2(\Omega)$.

\begin{lmm}\label{globalsol}
	Problem (\ref{cost})--(\ref{system}) has at least one global solution.
\end{lmm}

In order to make notation simpler, from now on we  fix a local solution $\bar u\in\mathcal U$ of problem 
(\ref{cost})--(\ref{system}). We call the function $H:\Omega\times\mathbb R\times\mathbb R\times\mathbb R\to\mathbb R$, given by
\begin{align*}
H(x,y,p,u) := g(x,y,u)+ p f(x,y,u),
\end{align*}
the  Hamiltonian of problem (\ref{cost})--(\ref{system}).
Given $u\in\mathcal U$, we say that $p_u\in D(\mathcal L)$ is the costate associated with $u\in\mathcal 
U$ if 
$$
\mathcal Lp_u=H_y(\cdot,y_u,p_u,u).
$$
The following proposition shows that the {mapping} $u\to p_u$ from 
$\mathcal{U}$ to $D(\mathcal L)$ is well defined.
We {give} the proof of this elementary result because it {seems not to be} explicitly stated in the literature. 
\begin{prpstn}\label{concost}
	For each $u\in\mathcal U$ there exists a unique costate $p_u\in D(\mathcal L)$ associated with $u\in\mathcal U$. Moreover, $\{p_u:u\in\mathcal U\}$ is a bounded subset of $H^1(\Omega)\cap C(\bar\Omega)$ and for each $r>n/2$ there exist $c>0$ such that
	\begin{align*}
	|p_{u_2}-p_{u_1}|_{H^1(\Omega)\cap C(\bar\Omega)}\le c|u_2-u_1|_{L^r(\Omega)}
	\end{align*}
	for all $u_1,u_2\in\mathcal U$. 
\end{prpstn}
\begin{proof}
	The existence and uniqueness follows from Lemma \ref{L1e}. 
	Given $u\in\mathcal U$, the function $p_u$ satisfies 
	\begin{align*}
	\mathcal{L}p_{u}+d_{y}(\cdot,y_u)p_u=g_{y}(\cdot,y_u,u).
	\end{align*}
	By $(ii),(iii)$ and $(iv)$ of Assumption \ref{A1}, for each $u\in\mathcal U$, the function $d_y(\cdot,y_{u})$ 
	is nonnegative and belongs to $L^\infty(\Omega)$. By $(ii)$ and $(iii)$ of Assumption \ref{A1}, for each $u\in\mathcal U$ 
	the function $g_{y}(\cdot,y_u,u)$ belongs to $L^\infty(\Omega)$. Furthermore, since by Proposition \ref{contstate}
	the set $\{y_{u}: u\in\mathcal U\}$ is bounded in $C(\bar\Omega)$, there exists $M_1>0$ such that 
	\begin{align*}
	|g_{y}(\cdot,y_u,u)|_{L^\infty(\Omega)}\le M_1
	\end{align*}
	for all $u\in\mathcal U$. By Lemma \ref{L1e}, there exists a positive number $c_1$ such that for all $u\in\mathcal U$
	\begin{align*}
	|p_{u}|_{H^1(\Omega)\cap C(\bar\Omega)}\le c_1|g_{y}(\cdot,y_u,u)|_{L^\infty(\Omega)}.
	\end{align*}
	Thus, $M_2:=c_1M_1$ is a bound for the set $\{p_u:u\in\mathcal U\}$ in $H^1(\Omega)\cap C(\bar\Omega)$. 
	Let $u_1,u_2\in\mathcal U$ and  $r>n/2$. We have then
	\begin{align*}
	\mathcal L(p_{u_2}-p_{u_1})+d_{y}(\cdot,y_{u_2})(p_{u_2}-p_{u_1})=
	H_{y}(\cdot,y_{u_2},p_{u_1},u_2)-H_{y}(\cdot,y_{u_1},p_{u_1},u_1).
	\end{align*}
	By Lemma \ref{L1e}, there exists a positive number $c_2$ (independent of $u_1$ and $u_2$) such that 
	\begin{align*}
	|p_{u_2}-p_{u_1}|_{H^1(\Omega)\cap C(\bar\Omega)}\le c_2|H_{y}(\cdot,y_{u_2},p_{u_1},u_2)-
	H_{y}(\cdot,y_{u_1},p_{u_1},u_1)|_{L^r(\Omega)}.
	\end{align*}
	
	By $(ii)$ of Assumption \ref{A1} and the boundedness of the set $\{p_u:u\in\mathcal U\}$ in $C(\bar\Omega)$, 
	there exists $L>0$ such that 
	\begin{align*}
	|H_{y}(\cdot,y_{u_2},p_{u_1},u_2)-H_{y}(\cdot,y_{u_1},p_{u_1},u_1)|
	\le L\Big(|y_{u_2}-y_{u_1}|+|u_2-u_1|\Big)\quad\text{a.e. in $\Omega$}.
	\end{align*}
	Consequently,
	\begin{align*}
	|p_{u_2}-p_{u_1}|_{H^1(\Omega)\cap C(\bar\Omega)} &\le 
	{  c_2 L \big(| y_{u_1} - y_{u_2} |_{L^r(\Omega)} + | u_1 - u_2 |_{L^r(\Omega)}\big)}
	\\ & \le c_2L \Big((\text{meas}
	\hspace*{0.05cm}\Omega) ^\frac{1}{r}|y_{u_2}-y_{u_1}|_{L^\infty(\Omega)}+|u_2-u_1|_{L^r(\Omega)}\Big).
	\end{align*}
	By Proposition \ref{contstate},  there exists a constant $c_3>0$ (independent of $u_1$ and $u_2$) such that
	\begin{align*}
	|y_{u_2}-y_{u_1}|_{C(\bar\Omega)}\le c_3|u_2-u_1|_{L^r(\Omega)}.
	\end{align*}
	Thus,
	\begin{align*}
	|p_{u_2}-p_{u_1}|_{H^1(\Omega)\cap C(\bar\Omega)}\le c_2L\big(1+c_3 (\text{meas}
	\hspace*{0.05cm}\Omega)^\frac{1}{r} \big)|u_2-u_1|_{L^r(\Omega)}.
	\end{align*}
	The estimate follows defining $c:=c_2L\big(1+c_3 (\text{meas}\hspace*{0.05cm}\Omega)^\frac{1}{r} \big)$.
\end{proof}

We call the function $\mathcal S:\mathcal U\to H^1(\Omega)\cap C(\bar\Omega)$ given by $\mathcal S(u):= p_u$ 
the control-to-adjoint mapping. The following proposition gives us another useful estimate; it can be easily proved employing Lemma \ref{L2e} and the argument in the  proof of \cite[Theorem 4.16]{TroeltzschPde}. 

\begin{prpstn}\label{contstatel1}
	There exists $c>0$ such that 
	\begin{align*}
	|y_{u_2}-y_{u_1}|_{L^2(\Omega)}+|p_{u_2}-p_{u_1}|_{L^2(\Omega)}\le c|u_2-u_1|_{L^1(\Omega)}
	\end{align*}
	for all $u_1,u_2\in\mathcal U$. 
\end{prpstn}

We close this subsection with the following result.

\begin{prpstn}\label{weakconvergence}
	Let $\{u_m\}_{m=1}^\infty$ be a sequence in $\mathcal U$ and $u\in \mathcal U$.  If $u_m\rightharpoonup u$ 
	weakly in $L^2(\Omega)$, then $y_{u_m}\to y_{u}$ and $p_{u_m}\to p_{u}$ in $C(\bar\Omega)$.
\end{prpstn}

\begin{proof}
	We prove only the convergence $p_{u_m}\to p_u$ in $C(\bar\Omega)$, the convergence 
	$y_{u_m}\to y_u$ in $C(\bar\Omega)$ is analogous.
	Let $\{p_{u_{m_k}}\}_{k=1}^\infty$ be an arbitrary subsequence of   $\{p_{u_m}\}_{m=1}^\infty$. 
	By the compact embedding $H^1(\Omega)\hookrightarrow L^2(\Omega)$, there exists a subsequence
	of $\{p_{u_{m_k}}\}_{k=1}^\infty$, denoted in the same way, and $p\in L^2(\Omega)$ such that 
	$p_{u_{m_k}}\to p$ in $L^2(\Omega)$. Since $y_{u_{m_k}}\to y_{u}$ in $C(\bar\Omega)$, one can deduce that
	\begin{align*}
	H_{y}(\cdot,y_{u_{m_k}}, p_{u_{m_k}}, {u_{m_k}})\rightharpoonup H_{y}(\cdot,y_u,p,u) \quad\text{weakly in $L^2(\Omega)$.}
	\end{align*}
	By Lemma \ref{weakconv}, we have $p_{u_{m_{k}}}\to p_u$ in $C(\bar\Omega)$. The result follows, since every 
	subsequence of $\{p_{u_{m}}\}_{m=1}^\infty$ has  a {further} subsequence that converges to $p_{u}$ in $C(\bar\Omega)$.
\end{proof}


\section{Differentiability of the mappings involved}\label{Section diff}

In this section, we prove some preliminary results concerning the differentiability of the control-to-state mapping, 
the control-to-adjoint mapping and the switching mapping (to be defined later).  Some of these properties 
are well known for the control-to-state mapping; see, e.g., 
\cite{Casascone,Casasbang,Casanum,Wachelliptic,TroeltzschPde}. Nevertheless, we require more specific 
estimates than the ones in the literature. The differentiability of the control-to-adjoint mapping and the switching mapping 
has not been studied before in the literature {on} elliptic control-constrained problems, therefore we devote this section 
to obtain appropriate estimates needed in the study of stability in the next section.

\subsection{The state and adjoint mappings}

We begin this subsection recalling the definition of directional derivative, see 
\cite[pp.2-4]{Flemming} or \cite[p.171]{Lue}. Let $Y$ be a normed space and ${\mathcal F}: {\mathcal U} \to Y$ a mapping. Given $u \in {\mathcal U}$ and $v\in\mathcal U-u$, if the limit
\begin{align*}
d {\mathcal F}(u;v) := \lim_{\varepsilon \to 0^+} \frac{{\mathcal F}(u + \varepsilon v) - {\mathcal F}(u)}{\varepsilon}
\end{align*}
exists in $Y$, we say that $\mathcal F(u;v)$ is the (G\^ateaux) differential of $\mathcal F$ at $u$ in the direction $v$. Note that by convexity of $\mathcal U$, $u + \varepsilon v$  belongs to $\mathcal U$ for every $u\in\mathcal U$, $v\in\mathcal U-u$ and  $\varepsilon \in [0,1]$. We will restrict ourselves to this simple definition of directional derivative, as further differentiability properties are not needed in our analysis of stability.

Recall that $\bar u\in \mathcal U$ is a fixed solution of problem (\ref{cost})--(\ref{system}). As it is well-known, 
the  differential of the control-to-state mapping at $\bar u$ is related to the linearization of the system equation 
around $\bar u$. Bearing this in mind, given $v\in L^2(\Omega)$, we denote by $z_{v}$ 
the unique\footnote{ The uniqueness follows from Lemma \ref{L1e}, and the fact that equation (\ref{linsta}) 
	can be rewritten as $$\mathcal Lz_{v}+d_{y}(\cdot,y_{\bar u})z_{v}=\beta(\cdot)v.$$} solution of the equation
\begin{align}\label{linsta}
\mathcal L z_{v}=f_y(\cdot, y_{\bar u},\bar u)z_{v}+f_u(\cdot,y_{\bar u},\bar u)v.
\end{align}

The proof of the following estimate can be found in the standard literature, see the proof of 
\cite[Theorem 4.17]{TroeltzschPde} for the case of a Neumann boundary problem (the proof is the same for
Robin or Dirichlet boundary). It can also be deduced by the same arguments given in the proof 
of Proposition \ref{Ecr}.

\begin{prpstn}\label{Esr}
	For each $r>n/2$ there exists $c>0$ such that
	\begin{align*}
	|y_u-y_{\bar u}-z_{u-\bar u}|_{H^1(\Omega)\cap C(\bar\Omega)}\le c|u-\bar u|_{L^r(\Omega)}^2\quad\quad 
	\forall u\in\mathcal U.
	\end{align*}
\end{prpstn}

One of the first things that can be deduced from Proposition \ref{Esr} is the differentiability of the control-to-state 
mapping $\mathcal G$. Given $v\in L^2(\Omega)$ satisfying $\bar u+v\in\mathcal U$, the  
differential of the control-to-state mapping $\mathcal G$ at $\bar u$ in the direction $v$ exists and  
is given by $d{\mathcal G}(\bar u;v)=z_{v}$. 
 For further differentiability properties of the control-to-state mapping, we refer the reader to \cite[Theorem 2.12]{Casasnonmonotone}. 

In order to study the  differential of the control-to-adjoint 
mapping{ we introduce the  following notations}. 
Given $v\in L^2(\Omega)$, we denote by $q_{v}$ the unique\footnote{ The uniqueness follows from Lemma \ref{L1e}, 
	and the fact that equation (\ref{lincos}) can be rewritten as 
	$$
	\mathcal Lq_{v}+d_{y}(\cdot,y_{\bar u})q_{v}=H_{yy}(\cdot,y_{\bar u}, p_{\bar u},\bar u)z_v+
	H_{yu}(\cdot,y_{\bar u}, p_{\bar u},\bar u)v.
	$$} 
solution of the equation
\begin{align}\label{lincos}
\mathcal Lq_{v}=H_{yy}(\cdot,y_{\bar u}, p_{\bar u},\bar u)z_v+H_{yp}(\cdot,y_{\bar u}, p_{\bar u},\bar u)q_{v}+
H_{yu}(\cdot,y_{\bar u}, p_{\bar u},\bar u)v.
\end{align}

The following estimate is concerned with the differentiability of the control-to-adjoint mapping. 
To the best of our knowledge, this result does not appear in the literature; therefore we present its proof, although it is standard. 
\begin{prpstn}\label{Ecr}
	For each $r>n/2$ there exists $c>0$ such that
	\begin{align*}
	|p_u-p_{\bar u}-q_{u-\bar u}|_{H^1(\Omega)\cap C(\bar\Omega)}\le c|u-\bar u|_{L^r(\Omega)}^2\quad
	\quad \forall u\in\mathcal U.
	\end{align*}
\end{prpstn}

\begin{proof}
	Given $u\in\mathcal U$, 
	we define $\psi_{u}:\Omega\to\mathbb R^4$ by $\psi_{u}(x):=(x,y_{u}(x),p_{u}(x),u(x))$. 
	For each $u\in \mathcal U$, we denote by $\tilde q_{u-\bar u}$ the unique solution of the equation
	\begin{align*}
	\mathcal L\tilde q_{u-\bar u}=H_{yy}(\psi_{\bar u})(y_{u}-y_{\bar u})+H_{yp}(\psi_{\bar u})\tilde q_{u-\bar u}+
	H_{yu}(\psi_{\bar u})(u-\bar u).
	\end{align*}
	Let $u\in\mathcal U$ and $r>n/2$ be arbitrary. Using the Taylor Theorem (integral form of the remainder) and $(ii)$-$(iii)$ of Assumption \ref{A1}, 
	one can find $\alpha_1,\alpha_2,\alpha_3\in L^\infty(\Omega)$ such that
	\begin{align*}
	H_{y}(\psi_{u})=&H_{y}(\psi_{\bar u})+H_{yy}(\psi_{\bar u})(y_{u}-y_{\bar u})+H_{yp}(\psi_{\bar u})(p_{u}-p_{\bar u})+H_{yu}(\psi_{\bar u})v\\
	&+\alpha_1(\cdot)(y_{u}-y_{\bar u})^2+\alpha_2(\cdot)(y_{u}-y_{\bar u})(p_{u}-p_{\bar u})+\alpha_3(\cdot)(y_{u}-y_{\bar u})v,
	\end{align*}
	where $v=u-\bar u$. Hence
	\begin{align*}
	\mathcal L(p_u-p_{\bar u}-\tilde q_{v})=H_{yp}(\psi_{\bar u})(p_{u}-p_{\bar u}-\tilde q_{v})+
	\Big[ \alpha_1(\cdot)(y_{u}-y_{\bar u})+\alpha_2(\cdot)(p_{u}-p_{\bar u})+\alpha_3(\cdot)v\Big](y_{u}-y_{\bar u}).
	\end{align*}
	By Lemma \ref{L1e}, Proposition \ref{contstate} and Proposition \ref{concost}, there exists $c_1>0$ such that
	\begin{align*}
	|p_u-p_{\bar u}-\tilde q_{v}|_{H^1(\Omega)\cap C(\bar\Omega)}\le c_1|v|_{L^r(\Omega)}^2.
	\end{align*}
	Now, 
	\begin{align*}
	\mathcal L(\tilde q_{v}- q_{v})=H_{yy}(\psi_{\bar u})(y_{u}-y_{\bar u}-z_{v})+H_{yp}(\psi_{\bar u})(\tilde q_{v}-q_{v}).
	\end{align*}
	By Lemma \ref{L1e} and Proposition \ref{Esr}, there exists $c_2>0$ such that 
	\begin{align*}
	|\tilde q_{v}- q_{v}|_{H^1(\Omega)\cap C(\bar\Omega)}\le c_2|v|_{L^r(\Omega)}^2.
	\end{align*}
	Finally, by the triangle inequality
	\begin{align*}
	|p_u-p_{\bar u}- q_{v}|_{H^1(\Omega)\cap C(\bar\Omega)}\le|p_u-p_{\bar u}-
	\tilde q_{v}|_{H^1(\Omega)\cap C(\bar\Omega)}+ |\tilde q_{v}- q_{v}|_{H^1(\Omega)\cap C(\bar\Omega)}.
	\end{align*}
	The result follows taking $c:=c_1+c_2$.
\end{proof}

Given $v\in L^\infty(\Omega)$ satisfying $\bar u+v\in\mathcal U$, the  differential of the control-to-adjoint 
mapping $\mathcal S$ at $\bar u$ in the direction $v$ exists and  is given by $d{\mathcal S}(\bar u;v)=q_{v}.$  

We now state further properties concerning the mappings $v\to z_v$ and $v\to q_v$.

\begin{prpstn}\label{furtherprop0}
	The following statements hold.
	\begin{itemize}
		\item[(i)] For each $r>n/2$ there exists a positive number $c$ such that 
		\begin{align*}
		|z_{v}|_{H^1(\Omega)\cap C(\bar\Omega)}+|q_{v}|_{H^1(\Omega)\cap C(\bar\Omega)}\le c|v|_{L^r(\Omega)}\quad\forall v\in L^2(\Omega)\cap L^r(\Omega).
		\end{align*}
		\item[(ii)] There exists a positive number $c$ such that 
		\begin{align*}
		|z_{v}|_{L^2(\Omega)}+|q_{v}|_{L^2{(\Omega)}}\le c|v|_{L^1(\Omega)} \quad \forall v\in L^2(\Omega).
		\end{align*}
		\item[(iii)] Let $\{v_k\}_{k=1}^\infty$ be a sequence in $L^2(\Omega)$ and $v\in L^2(\Omega)$. If $v_k\rightharpoonup v$ weakly in $L^2(\Omega)$, then $z_{v_{k}}\to z_{v}$ and $q_{v_{k}}\to q_{v}$ in $C(\bar\Omega)$. 
	\end{itemize}
\end{prpstn}

\begin{proof}
	Items $(i)$ and $(ii)$ follow from Lemma \ref{L1e} and \ref{L2e}, respectively. Item $(iii)$ follows from 
	Lemma \ref{weakconv}.
\end{proof}

\subsection{The switching mapping}

Let us begin this subsection recalling the first order necessary condition (Pontryagin principle {in integral form}) for problem 
(\ref{cost})--(\ref{system}). If $u\in\mathcal U$ is a local solution of problem (\ref{cost})--(\ref{system}), then
\begin{align}\label{varin}
\int_{\Omega}\Big[s(x,y_u)+\beta(x)p_u\Big](w-u)\, dx\ge0\quad\forall w\in\mathcal U.
\end{align}
The variational inequality (\ref{varin}) motivates the  following definition. For each $u\in\mathcal U$, define
\begin{align*}
\sigma_{u}:= s(\cdot,y_u)+\beta(\cdot)p_u.
\end{align*}
Observe that $\sigma_u=H_{u}(\cdot,y_u,p_u)$. The mapping $\mathcal Q:\mathcal U\to L^\infty(\Omega)$ given by $\mathcal Q(u):=\sigma_u$ is called 
the switching mapping. Given $v\in L^2(\Omega)$, we define 
\begin{align*}
\pi_{v}:=H_{uy}(\cdot,y_{\bar u},p_{\bar u})z_{v}+H_{up}(\cdot,y_{\bar u},p_{\bar u})q_v.
\end{align*}
{This definition is justified by the following estimate.}

\begin{prpstn}\label{sigmatilde}
	For each $r>n/2$ there exists $c>0$ such that
	\begin{align*}
	|\sigma_u-\sigma_{\bar u}-\pi_{u-\bar u}|_{L^\infty(\Omega)}\le c|u-\bar u|_{L^r(\Omega)}^2\quad\quad \forall u\in\mathcal U.
	\end{align*}
\end{prpstn}

\begin{proof}
	Given $u\in\mathcal U$, we define $\psi_{u}:\Omega\to\mathbb R^3$ by $\psi_{u}(x):=(x,y_{u}(x),p_{u}(x))$. 
	For each $u\in \mathcal U$, we {denote}
	\begin{align*}
	\tilde\pi_{u-\bar u}:=H_{uy}(\psi_{\bar u})(y_{u}-y_{\bar u})+H_{up}(\psi_{\bar u})(p_u-p_{\bar u}).
	\end{align*}
	Let $u\in\mathcal U$ and $r>n/2$ be arbitrary, and abbreviate $v=u-\bar u$.  Using the Taylor Theorem (integral form of the remainder) and $(ii)$-$(iii)$ of Assumption \ref{A1}, one can find $\alpha\in L^\infty(\Omega)$ such that
	\begin{align*}
	H_{u}(\psi_{u})=&H_{u}(\psi_{\bar u})+H_{uy}(\psi_{\bar u})(y_{u}-y_{\bar u})+H_{up}(\psi_{\bar u})(p_{u}-p_{\bar u})+\alpha(\cdot)(y_{u}-y_{\bar u})^2.
	\end{align*}
	Therefore, by Proposition \ref{contstate}, there exists $c_1>0$ such that
	\begin{align*}
	|\sigma_u-\sigma_{\bar u}-\tilde \pi_{v}|_{L^\infty(\Omega)}\le c_1|v|_{L^r(\Omega)}^2.
	\end{align*}
	Now, 
	\begin{align*}
	|\tilde \pi_{v}-\pi_v|_{L^\infty(\Omega)}\le |H_{uy}(\cdot,y_{\bar u},p_{\bar u})(y_{u}-y_{\bar u}-z_{v})+
	H_{up}(\cdot,y_{\bar u},p_{\bar u})(q_u-q_{\bar u}-q_v)|_{L^\infty(\Omega)}.
	\end{align*}
	Hence, by Proposition \ref{Esr} and \ref{Ecr}, there exists $c_2>0$ such that 
	\begin{align*}
	|\tilde \pi_{v}-\pi_v|_{L^\infty(\Omega)}\le c_2|v|_{L^r(\Omega)}^2.
	\end{align*}
	Finally, by the triangle inequality,
	\begin{align*}
	|\sigma_u-\sigma_{\bar u}- \pi_{v}|_{L^\infty(\Omega)}\le|\sigma_u-\sigma_{\bar u}-\tilde \pi_{v}|_{L^\infty(\Omega)}
	+|\tilde \pi_{v}-\pi_v|_{L^\infty(\Omega)}.
	\end{align*}
	The result follows defining $c:=c_1+c_2$.
\end{proof}

Proposition \ref{sigmatilde} yields immediately that the  differential of the switching mapping 
$\mathcal Q$ at $\bar u$ in any direction $v\in\mathcal U-\bar u$ exists and is given by $d{\mathcal Q}(\bar u;v)=\pi_{v}.$

One of the important features of the mapping $v\to\pi_v$ is the following.

\begin{prpstn}\label{opQ}
	For all $v\in L^2(\Omega)$, we have
	\begin{align*}
	\int_{\Omega} \pi_{v}v\, dx=\int_{\Omega}\Big[ H_{yy}(x,y_{\bar u},p_{\bar u},\bar u)z_{v}^2+
	2H_{uy}(x,y_{\bar u},p_{\bar u},\bar u)z_{v} v \Big]\,dx.
	\end{align*}
\end{prpstn}

\begin{proof}
	In order to simplify notation, we write $\psi_{\bar u}(x):=(x,y_{\bar u}(x),p_{\bar u}(x),\bar u(x))$ for 
	each $x\in\Omega$. Let $v\in L^2(\Omega)$ be arbitrary. By the integration by parts formula (\ref{intbyparts}) and the concrete form of the Hamiltonian, we get
	\begin{align*}
	\int_{\Omega} H_{up}(\psi_{\bar u})q_{v}v\, dx&=\int_{\Omega} \big(\mathcal L z_v+d_y(x,y_{\bar u})z_v\big)q_v\,dx
	=\int_{\Omega}\big(\mathcal L q_v+d_y(x,y_{\bar u})q_v\big)z_v\, dx\\
	&=\int_{\Omega}\big(H_{yy}(\psi_{\bar u})z_v+H_{uy}(\psi_{\bar u})v\big)z_v=
	\int_{\Omega}\Big[H_{yy}(\psi_{\bar u})z_v^2+H_{uy}(\psi_{\bar u})z_v v\Big]\, dx.
	\end{align*}
	The result follows since
	\begin{align*}
	\int_{\Omega} \pi_{v}v\, dx=\int_{\Omega}H_{uy}(\psi_{\bar u})z_v v\, dx+
	\int_{\Omega}H_{up}(\psi_{\bar u})q_v v\, dx.
	\end{align*}
\end{proof}

We give further properties of the mapping $v\to\pi_v$ in the next proposition, its proof follows trivially from Proposition 
\ref{furtherprop0}.

\begin{prpstn}\label{furprop3}
	The following statements hold.
	\begin{itemize}
		\item[(i)] For each $r>n/2$ there exists a positive number $c$ such that 
		\begin{align*}
		|\pi_{v}|_{L^\infty(\Omega)}\le c|v|_{L^r(\Omega)}\quad\forall v\in L^2(\Omega)\cap L^r(\Omega).
		\end{align*}
		\item[(ii)] There exists a positive number $c$ such that 
		\begin{align*}
		|\pi_{v}|_{L^2(\Omega)}\le c|v|_{L^1(\Omega)} \quad \forall v\in L^2(\Omega).
		\end{align*}
		\item[(iii)] Let $\{v_k\}_{k=1}^\infty$ be a sequence in $L^2(\Omega)$ and $v\in L^2(\Omega)$.
		If $v_k\rightharpoonup v$ weakly in $L^2(\Omega)$, then $\pi_{v_{k}}\to \pi_{v}$ in $L^\infty(\Omega)$. 
	\end{itemize}
\end{prpstn}

Proposition \ref{opQ} motivates the following definition.
For each $v\in L^2(\Omega)$, define
\begin{align}\label{quadraticform}
\Lambda(v):=\int_{\Omega}\Big[ H_{yy}(x,y_{\bar u},p_{\bar u},\bar u)z_{v}^2+
2H_{uy}(x,y_{\bar u},p_{\bar u},\bar u)z_v v \Big]\,dx.
\end{align}

\begin{rmrk}
	We mention that the quadratic form $\Lambda:L^2(\Omega)\to\mathbb R$ is the second variation of the objective
	functional $\mathcal J:\mathcal U\to\mathbb R$ at $\bar u$. By Proposition \ref{opQ}, we also have 
	the following representation
	\begin{align*}
	\Lambda(v)=\int_\Omega \pi_{v}v\,dx\quad\forall v\in L^2(\Omega).
	\end{align*}
\end{rmrk}

We close this section with a result concerning the quadratic form (\ref{quadraticform}). 

\begin{prpstn}\label{lemcas}
	Let $\{v_k\}_{k=1}^\infty\subset L^2(\Omega)$ and $v\in L^2(\Omega)$. If $v_k\rightharpoonup v$ 
	weakly in $L^2(\Omega)$, then $\Lambda(v_k)\to \Lambda(v)$. 
\end{prpstn}

\begin{proof}
	By Proposition \ref{furprop3}, $\pi_{v_k}\to\pi_v$ in $L^\infty(\Omega)$, therefore
	\begin{align*}
	\Lambda(v_k)=\int_\Omega (\pi_{v_k}-\pi_v)v_k\,dx+\int_\Omega \pi_{v}v_k\,dx \to \int_\Omega \pi_{v}v\,dx.
	\end{align*}
\end{proof}

\section{Stability}\label{Section Stab}
In this section, we study the stability of the optimal solution
of problem (\ref{cost})--(\ref{system}) with respect to perturbations.
As usual in optimization, the stability of the solution is derived
from stability of the system of necessary optimality conditions.
The investigated stability property of the latter is the so-called
strong metric Hölder subregularity (SMHSr), see e.g., \cite[Section 3I]{Dontchevbook} or \cite[Section 4]{Subreg}.
After introducing the assumptions we study the SMHSr property
of the variational inequality (9). Then the result is used to obtain
this property for the whole system of necessary optimality conditions

\subsection{The main assumption}
We begin the section recalling that  
$\bar u\in\mathcal U$ is a local minimizer of problem (\ref{cost})--(\ref{system}), and the definition of the quadratic form 
$\Lambda: L^2(\Omega)\to\mathbb R$ in (\ref{quadraticform}).
\begin{ssmptn}\label{A2}
	There exist positive numbers $\alpha_0,\gamma_0$ and $k^*\in[1,4/n)$ such that
	\begin{align}\label{asscond}
	\int_{\Omega}\sigma_{\bar u}(u-\bar u)\,dx+\Lambda(u-\bar u)\ge\gamma_0|u-\bar u|_{L^1(\Omega)}^{{k^*}+1},
	\end{align}
	for all $u\in\mathcal U$ with $|u-\bar u|_{L^1(\Omega)}\le\alpha_0$.
\end{ssmptn}
Assumption \ref{A2} resembles the well-known $L^2$-coercivity condition in optimal control, with two substantial differences: $(i)$ the left-hand side of (\ref{asscond}) involves a linear term (not only the quadratic form in the $L^2$-coercivity condition); $(ii)$ the $L^1$-norm appears in the right-hand side of (\ref{asscond}). We mention that the standard $L^2$-coercivity condition cannot hold in affine problems. 
Assumption \ref{A2} in the particular case ${k^*}=1$ has been used before in the literature on optimal control 
problems constrained by ordinary differential equations, see \cite[Assumption A2']{SubregOsm} or 
\cite[Assumption A2]{Mayersubreg}. A similar assumption was used in \cite[Assumption 2]{Hoeldersubreg}.
We first point out that if $\bar u$ satisfies Assumption \ref{A2}, then it must be bang-bang. A control $u\in\mathcal U$ is bang-bang  
if $ u(x)\in\{b_1(x),b_2(x)\}$ for a.e. $x$ in $\Omega$. The proof of this result 
follows the arguments given in the proof of \cite[Theorem 2.1]{Casanum}.

\begin{prpstn}\label{bangbang}
	If $\bar u\in \mathcal U$ satisfies Assumption \ref{A2}, then $\bar u$ is bang-bang.
\end{prpstn}

\begin{proof}
	Let $\alpha_0$ and $\gamma_0$ be the positive numbers in Assumption \ref{A2}. 
	Suppose that there exists $\varepsilon>0$ and  a measurable set $E\subset\Omega$ of positive measure such that 
	\begin{align*}
	\bar u(x)\in [b_1(x)+\varepsilon,b_2(x)-\varepsilon]\quad \text{for a.e. $x\in E$.}
	\end{align*}
	Define $\varepsilon^*:=\min\{\alpha_0 (\text{meas}\hspace*{0.05cm} E)^{-1}, \varepsilon\}$. 
	Let $\{v_m\}_{m=1}^\infty\subset L^2(\Omega)$ be a sequence converging to zero weakly in $L^2(\Omega)$ 
	such that for each $m\in\mathbb N$,  $v_{m}(x)\in\{-\varepsilon^*,\varepsilon^*\}$ for a.e. $x\in\Omega$. 
	For each $m\in \mathbb N$, define
	\begin{align*}
	u_{m}(x):=\left\{ \begin{array}{lcc}
	\bar u(x) &   if  & x\notin E \\
	\\ \bar u(x)+v_{m}(x)&  if & x\in E.
	\end{array} \right.
	\end{align*}
	Clearly, for each $m\in\mathbb N$, $u_m$ belongs to $\mathcal U$ and
	\begin{align*}
	|u_{m}-\bar u|_{L^1(\Omega)}=\varepsilon^* \hspace*{0.05cm}\text{meas}\hspace*{0.05cm} E.
	\end{align*} 
	Hence, by Assumption \ref{A2}
	\begin{align}\label{lhs}
	\int_{\Omega}\sigma_{\bar u}(u_m-\bar u)\,dx+\Lambda(u_m-\bar u)\ge\gamma_0\Big(\varepsilon^*
	\hspace*{0.05cm}\text{meas}\hspace*{0.05cm} E\Big)^{{k^*}+1}
	\end{align}
	for all $m\in\mathbb N$. Since $u_m\rightharpoonup \bar u$ weakly in $L^2(\Omega)$, we have 
	by {Proposition} \ref{lemcas} that the left hand side of $(\ref{lhs})$ converges to $0$; a contradiction.
\end{proof}

Proposition \ref{bangbang} makes the following lemma relevant. The proof follows the argument used in the proof of \cite[Theorem 4.4]{BangBangconvergence}. Alternatively, as argued in the proof of \cite[Theorem 4.3]{Wachelliptic}, one can also use \cite[Theorem 1]{Visin} and the fact that for a.e. $x\in\Omega$, $u(x)$ is an extremal point of  $\overline{\text{conv}}(\{u_{k}(x)\}_{k=1}^\infty\cup u(x))$ if $u\in\mathcal U$ is bang-bang.

\begin{lmm}\label{weakimpliesstrong}
	Let $u\in\mathcal U$ be bang-bang, and $\{u_{k}\}_{k=1}^\infty\subset \mathcal U$ be a sequence.
	If $u_k\rightharpoonup u$ weakly in $L^1(\Omega)$, then $u_k\to u$ in $L^1(\Omega)$.
\end{lmm}

\begin{proof}
	Let $\Omega_i:=\{x\in\Omega:u(x)=b_i(x)\}$, $i=1,2$. Let $\chi_{\Omega_i}:\Omega\to\{0,1\}$ denote the characteristic function of the set $\Omega_i$, $i=1,2$. Now, by definition of weak convergence
	\begin{align*}
	\int_{\Omega}|u_k-u|\, dx=\int_{\Omega}\chi_{\Omega_1}(u_n-\bar u)\, dx-\int_{\Omega}\chi_{\Omega_2}(u_n-\bar u)\, dx\to 0.
	\end{align*}
\end{proof}

The next proposition shows that the switching  mapping satisfies a growth condition. 
The proof consists of two steps. The first one is to show that Assumption \ref{A2} implies this growth 
condition for the linearization of the switching mapping. The second step is to adequately use the linearization 
as an  approximation of the switching mapping. 

\begin{prpstn}\label{propofequiv}
	Let Assumption 2 be fulfilled. Then there exist positive numbers $\alpha$ and $\gamma$ such that
	\begin{align*}
	\int_{\Omega} \sigma_{u}(u-\bar u)\,dx\ge \gamma |u-\bar u|_{L^1(\Omega)}^{k^*+1}
	\end{align*}
	for all $ u\in\mathcal U$ with $|u-\bar u|_{L^1(\Omega)}\le\alpha$.
\end{prpstn}

\begin{proof}{}
	Let $\alpha_0,\gamma_0$ and $k^*$ be the positive numbers in Assumption \ref{A2}. Fix $r\in(n/2,2/{k^*})$. 
	Using Proposition \ref{sigmatilde}, a constant $c>0$ can be found such that
	\begin{align}\label{inevarlr}
	|\sigma_u-\sigma_{\bar u}-\pi_{u-\bar u}|_{L^\infty(\Omega)}\le c|u-\bar u|_{L^1(\Omega)}^{2/r}\quad\quad 
	\forall u\in\mathcal U.
	\end{align}
	From Proposition \ref{opQ} and Assumption \ref{A2}, we have
	\begin{align}\label{intpr}
	\int_{\Omega}\Big[\sigma_{\bar u}+\pi_{u-\bar u}\Big](u-\bar u)\, dx=
	\int_{\Omega}\sigma_{\bar u}(u-\bar u)\,dx+\Lambda(u-\bar u)\ge \gamma_0 |u-\bar u|_{L^1(\Omega)}^{{k^*}+1}
	\end{align}
	for all $u\in\mathcal U$ with $|u-\bar u|_{L^1(\Omega)}\le\alpha_0$. Define $\gamma:=\gamma_0/2$  and  
	\begin{align*}
	\alpha:=\min\left\lbrace \alpha_0,\gamma^{\frac{r}{2-{k^*}r}} c^{-{\frac{r}{2-{k^*}r}}}\right\rbrace.
	\end{align*}
	Then, by (\ref{inevarlr})
	\begin{align}\label{inevarlr1}
	|\sigma_u-\sigma_{\bar u}-\pi_{u-\bar u}|_{L^\infty(\Omega)}\le c|u-\bar u|_{L^1(\Omega)}^{\frac{2}{r}}= 
	c|u-\bar u|_{L^1(\Omega)}^{\frac{2}{r}-{k^*}}|u-\bar u|_{L^1(\Omega)}^{k^*}\le 
	\gamma|u-\bar u|_{L^1(\Omega)}^{k^*}
	\end{align}
	for all $u\in\mathcal U$ with $|u-\bar u|_{L^1(\Omega)}\le\alpha$. We have for all $u\in\mathcal U$
	\begin{align*}
	\int_{\Omega}\sigma_{u}(u-\bar u)\,dx&=\int_{\Omega} \Big[\sigma_{\bar u}+\pi_{u-\bar u}\Big](u-\bar u)\,dx+\int_{\Omega} \Big[\sigma_{u}-\sigma_{\bar u}-\pi_{u-\bar u}\Big](u-\bar u)\,dx.
	\end{align*}
	Consequently, by  (\ref{intpr}) and (\ref{inevarlr1}), 
	\begin{align*}
	\int_{\Omega}\sigma_{u}(u-\bar u)\,dx&\ge\gamma_0 |u-\bar u|_{L^1(\Omega)}^{{k^*}+1}-|
	\sigma_{u}-\sigma_{\bar u}-\pi_{u-\bar u}|_{L^\infty(\Omega)}|u-\bar u|_{L^1(\Omega)}\\
	&\ge(\gamma_0-\gamma)|u-\bar u|_{L^1(\Omega)}^{{k^*}+1}=\gamma|u-\bar u|_{L^1(\Omega)}^{{k^*}+1}
	\end{align*}
	for all $u\in\mathcal U$ with $|u-\bar u|_{L^1(\Omega)}\le\alpha$.
\end{proof}

\subsection{Some existence and stability results}

We now pass to some preparatory lemmas concerning the existence of solutions of inclusions (also called generalized equations, see \cite{Robinsongenequ1}) related to 
the first order necessary condition of problem (\ref{cost})--(\ref{system}). Given $r\in[1,\infty]$, we denote by 
$\mathbb B_{L^r}(c;\alpha)$ the closed ball in $L^r(\Omega)$ with center $c\in L^r(\Omega)$ and radius $\alpha>0$.

The variational inequality (\ref{varin}) can be written as the inclusion
\begin{align}\label{Inclusion}
	0\in\sigma_u+ N_{\mathcal U}(u),
\end{align}
where the normal cone at $u$ to the set $\mathcal U$ is given by
\begin{align*}
	N_{\mathcal U}(u)=\left\lbrace \sigma\in L^\infty(\Omega): \int_{\Omega}\sigma(w-u)\, dx\le 0\quad \forall w\in\mathcal U\right\rbrace.
\end{align*}

\begin{lmm}\label{Lemainclusion0}
	For all $\rho\in {L^\infty(\Omega)}$ and $\varepsilon>0$ there exists 
	$u\in \mathcal U\cap \mathbb B_{L^1}(\bar u;\varepsilon)$ satisfying 
	$$
	\rho\in\sigma_u+ N_{\mathcal U\cap \hspace*{0.02cm}\mathbb B_{L^1}(\bar u;\varepsilon)}(u).
	$$
\end{lmm}

\begin{proof}
	Let $\rho\in {L^\infty(\Omega)}$ and $\varepsilon>0$. Consider the functional $\mathcal J_{\rho}:\mathcal U\cap \mathbb B_{L^1}(\bar u;\varepsilon)\to\mathbb R$ given by
	\begin{align*}
	\mathcal J_{\rho}(u):=\int_{\Omega} \big[g(y_u,u)-\rho u\big]\, dx=\mathcal J(u)-\int_{\Omega}\rho u\,dx.
	\end{align*}
	The functional $\mathcal J_\rho$ has at least one global minimizer $ u_\rho\in\mathcal U\cap \mathbb B_{L^1}(\bar u;\varepsilon)$ since 
	$\mathcal U\cap \mathbb B_{L_1}(\bar u;\varepsilon)$ is a weakly sequentially compact subset of $L^2(\Omega)$ 
	and $\mathcal J_{\rho}$ is weakly sequentially continuous. By the Pontryagin principle, 
	\begin{align*}
	\int_{\Omega}\big[\sigma_{u_\rho}-\rho\big](u-u_\rho)\, dx\ge0\quad \forall u\in\mathcal U\cap \mathbb B_{L^1}(\bar u;\varepsilon).
	\end{align*}
	We have then that $u_\rho$ satisfies 
	$\rho\in\sigma_{u_\rho}+ N_{\mathcal U\cap \hspace*{0.02cm}\mathbb B_{L^1}(\bar u;\varepsilon)}(u_\rho)$.
\end{proof}

\begin{lmm}\label{coneint}
	Let $\mathcal V_1$ and $\mathcal V_2$ be closed and convex subsets of $L^1(\Omega)$ 
	such that $\mathcal V_1\cap\text{int}\hspace*{0.07cm}\mathcal V_2\neq\emptyset$. Then
	\begin{align}\label{EdecomN}
	N_{\mathcal V_1\cap\mathcal V_2}(u)= N_{\mathcal V_1}(u)+N_{\mathcal V_2}(u)
	\end{align}
	for all $u\in\mathcal V_1\cap\mathcal V_2$.
\end{lmm}

\begin{proof}
	Given a set $\mathcal W \subset L^1(\Omega)$,  let $s_\mathcal W : L^\infty(\Omega) \to \mathbb R\cup\{+\infty\}$ 
	denote the support function to $\mathcal W$, that is 
	\begin{align*}
	s_\mathcal W(h) := \sup_{w \in \mathcal W} \int_{\Omega}hw\, dx.
	\end{align*}
	By \cite[Proposition 3.1]{Coneintersection}, the set 
	$\text{Epi}\hspace{0.07cm} s_{\mathcal V_1} + \text{Epi}\hspace{0.07cm} s_{\mathcal V_2}$ is a weakly${}^*$ 
	closed subset of $L^\infty(\Omega)$.
	{Then the representation (\ref{EdecomN}) holds according 
		to \cite[Theorem 3.1]{Coneintersection}.}
\end{proof}

We can now prove existence of solutions of the inclusion $\rho\in\sigma_u+N_{\mathcal U}(u)$ 
that are close (in the $L^1$-norm) to $\bar u$ whenever $\rho$ is close to zero 
(in the norm $L^\infty$-norm). The proof follows the arguments in \cite[p. 1127]{Bim19}.

\begin{lmm}\label{exisrho}
	Let Assumption \ref{A2} hold. Then for each $\varepsilon>0$ there exists $\delta>0$ such that for each 
	$\rho\in \mathbb B_{L^\infty}(0;\delta)$ there exists $u\in \mathcal U\cap \mathbb B_{L^1}(\bar u;\varepsilon)$ 
	satisfying $\rho\in\sigma_u+ N_{\mathcal U}(u)$.
\end{lmm}

\begin{proof}
	Let $\alpha$ and $\gamma$ be the numbers in Proposition \ref{propofequiv}. 
	Define $\varepsilon_0:=\min\{\varepsilon,\alpha\}$ and $\delta:=\varepsilon_0^{k^*}\gamma/2$. 
	Let $\rho\in L^\infty(\Omega)$ with $|\rho|_{L^\infty(\Omega)}\le \delta$. By Lemma \ref{Lemainclusion0}, 
	there exists $u\in\mathcal U\cap \mathbb B_{L^1}(\bar u;\varepsilon_0)$ such that 
	\begin{align*}
	\rho\in\sigma_u+ N_{\mathcal U\cap \mathbb B_{L^1}(\bar u;\varepsilon_0)}(u).
	\end{align*}
	Since trivially $\bar u\in \mathcal U\cap \text{int}\hspace*{0.04cm}\mathbb B_{L^1}(\bar u,\varepsilon_0)$, 
	by Lemma \ref{coneint}  we have
	\begin{align}\label{intcone}
	N_{\mathcal U\cap \mathbb B_{L^1}(\bar u;\varepsilon_0)}(u)= N_{\mathcal U}(u)+
	N_{\mathbb B_{L^1}(\bar u;\varepsilon_0)}(u).
	\end{align}
	Thus there exists $\nu\in N_{\mathbb B_{L^1}(\bar u;\varepsilon_0)}(u)$ such that 
	\begin{align*}
	\rho-\sigma_{u}-\nu\in N_{\mathcal U}(u).
	\end{align*}
	By definition of the normal cone,
	\begin{align}\label{nu0}
	0\ge\int_{\Omega}\big(\rho-\sigma_u\big)(\bar u-u)\,dx-\int_{\Omega}\nu(\bar u- u)\, dx.
	\end{align}
	As $\bar u\in \mathbb B_{L^1}(\bar u;\varepsilon_0)$ and $\nu\in N_{\mathbb B_{L^1}(\bar u;\varepsilon_0)}(u)$, we have  
	\begin{align*}
	\int_{\Omega}\nu(\bar u- u)\, dx\le 0.
	\end{align*}
	Consequently, by (\ref{nu0}) and Proposition \ref{propofequiv}
	\begin{align*}
	0\ge\int_{\Omega}\big(\rho-\sigma_u\big)(\bar u-u)\,dx\ge -|\rho|_{L^\infty(\Omega)}|u-
	\bar u|_{L^1(\Omega)}+\gamma|u-\bar u|_{L^1(\Omega)}^{{k^*}+1},
	\end{align*}
	which implies
	\begin{align*}
	|u-\bar u|_{L^1(\Omega)}\le \gamma^{-\frac{1}{{k^*}}}|\rho|_{L^\infty(\Omega)}^{\frac{1}{{k^*}}}
	\le2^{-\frac{1}{{k^*}}}\varepsilon_0<\varepsilon_0.
	\end{align*}
	As $u\in \text{int}\hspace*{0.07cm} \mathbb B_{L^1}(\bar u;\varepsilon_0)$, we have 
	$N_{\mathbb B_{L^1}(\bar u;\varepsilon_0)}(u)=\left\lbrace 0\right\rbrace $. Thus by (\ref{intcone}),
	\begin{align}
	\rho\in\sigma_u+ N_{\mathcal U\cap \mathbb B_{L^1}(\bar u;\varepsilon_0)}(u)=\sigma_u+ N_{\mathcal U}(u).
	\end{align}
\end{proof}

The following lemma shows how Proposition \ref{propofequiv} (and consequently Assumption \ref{A2}) is related to H\"{o}lder- stability.

\begin{lmm}\label{Esslem}
	Let Assumption \ref{A2} hold. There exist positive numbers $\alpha,\delta$ and $c$ such that for every $\rho\in \mathbb B_{L^\infty}(0;\delta) $ there exists $u\in\mathbb B_{L^1}(\bar u,\alpha)$  satisfying $\rho\in\sigma_u+ N_{\mathcal U}(u)$. Moreover,  \begin{align}\label{holderine}
	|u-\bar u|_{L^1(\Omega)}\le c|\rho|_{L^\infty(\Omega)}^{\frac{1}{{k^*}}}
	\end{align}
	for all  $\rho\in L^\infty(\Omega)$  and $u\in\mathbb B_{L^1}(\bar u;\alpha)$ satisfying $\rho\in\sigma_u+ N_{\mathcal U}(u)$.
\end{lmm}

\begin{proof}
	The existence part follows from Lemma \ref{exisrho}.
	Let $\alpha$ and $\gamma$ be the positive numbers in Proposition \ref{propofequiv}.
	Since $\rho-\sigma_u\in N_{\mathcal U}(u)$,	we have 
	\begin{align*}
	\int_{\Omega}(\rho-\sigma_{u})(\bar u-u)\,dx\le0.
	\end{align*}
	By Proposition \ref{propofequiv},
	\begin{align*}
	0&\ge\int_{\Omega}(\rho-\sigma_{u})(\bar u- u) 
	\,dx=\int_{\Omega} \sigma_{u}(u-\bar u)\,dx+
	\int_{\Omega}\rho(\bar u- u)\,dx\\
	&\ge \gamma\left( \int_{\Omega}|u-\bar u|\,dx\right)^{{k^*}+1}-|\rho|_{L^\infty(\Omega)}\int_{\Omega}|u-\bar u|\,dx.
	\end{align*}
	Hence
	\begin{align*}
	\displaystyle\int_{\Omega}|u-\bar u|\,dx\le\Big(\frac{1}{\gamma}|\rho|_{L^\infty(\Omega)}\Big)^{1/{k^*}}= 
	\gamma^{-\frac{1}{{k^*}}}|\rho|_{L^\infty(\Omega)}^{\frac{1}{{k^*}}}.
	\end{align*}
	The result follows defining $c=\gamma^{-\frac{1}{{k^*}}}$.
\end{proof}

For inequality (\ref{holderine}) to hold, Lemma \ref{Esslem} requires that the controls are close in the $L^1$-norm to the reference solution (by Lemma \ref{exisrho}, the existence of such controls is guaranteed). This closeness assumption on the controls can be removed if the solution of inclusion (\ref{Inclusion}) is unique. In particular, if  (\ref{Inclusion}) has a unique solution, then problem (\ref{cost})-(\ref{system}) has unique optimal control (minimizer).

\begin{lmm}\label{Esslemuni}
	Let Assumption \ref{A2} hold, and suppose additionally   $0\in \sigma_{ u}+N_{\mathcal U}(u)$ has a unique solution $\bar u\in\mathcal U$. There exist positive numbers $\delta$ and $c$ such that 
	\begin{align*}
	|u-\bar u|_{L^1(\Omega)}\le c|\rho|_{L^\infty(\Omega)}^{\frac{1}{{k^*}}}.
	\end{align*}
	for all  $\rho\in \mathbb B_{L^\infty}(0;\delta)$  and $u\in\mathcal U$  satisfying $\rho\in\sigma_u+ N_{\mathcal U}(u)$.
\end{lmm}

\begin{proof}
	Let $\alpha$ and $c$ be the positive numbers in Lemma \ref{Esslem}.
	First we prove that there exists $\delta>0$ such that if $u\in\mathcal U$ and $\rho\in L^\infty(\Omega)$  
	satisfy $\rho\in\sigma_u+N_{\mathcal U}(u)$ and $|\rho|_{L^\infty(\Omega)}\le\delta$, then 
	$u\in \mathbb B_{L^1}(\bar u;\alpha)$. Suppose not, then there exist sequences 
	$\{\rho_k\}_{k=1}^\infty\subset L^\infty(\Omega)$ and $\{u_{k}\}_{k=1}^\infty\subset \mathcal U$ 
	such that $\rho_k\in \sigma_{u_k}+ N_{\mathcal U}(u_k)$, $\rho_k\to 0$ in $L^\infty(\Omega)$, and $|u_k-\bar u|_{L^1(\Omega)}> \alpha$. 
	Since $\mathcal U$ is weakly sequentially compact in $L^2(\Omega)$, there exists a subsequence of 
	$\{u_k\}_{k=1}^\infty$, denoted in the same way, and $u^*\in\mathcal U$ such that $u_k\rightharpoonup u^*$ 
	weakly in $L^2(\Omega)$. Using Proposition \ref{weakconvergence},  one can see that  
	$\rho_k-\sigma_{u_k}\to \sigma_{u^*}$ in $L^\infty(\Omega)$.  Consequently, 
	as $\rho_k\in\sigma_{u_k}+ N_{\mathcal U}(u_k)$ for all $n\in\mathbb N$, 
	we obtain $0\in\sigma_{u^*}+N_{\mathcal U}(u^*)$. Then, by assumption, $u^*=\bar u$, so $u^*$ 
	is bang-bang. By Lemma \ref{weakimpliesstrong}, we have $u_k\to u^*$ in $L^1(\Omega)$; a contradiction. 
	The result follows from Lemma \ref{Esslem}.
\end{proof}

\subsection{Strong metric subregularity}

Let us begin considering the following system representing the necessary optimality conditions (Pontryagin principle) for problem (\ref{cost})--(\ref{system}):
\begin{align}\label{s1}
\left\{ \begin{array}{cll}
0&=&\mathcal Ly-f(\cdot,y,u),\\
0&=&\mathcal Lp-H_{y}(\cdot,y,p,u),\\
0&\in& H_{u}(\cdot,y,p) + N_{\mathcal U}(u),
\end{array} \right.
\end{align}
If $ u\in\mathcal U$ is a local solution of problem (\ref{cost})--(\ref{system}), then the triple $(y_{u}, p_{u},u)$ is a solution of (\ref{s1}).
	Therefore, the mapping that defines the right-hand side is referred to as the {\em optimality mapping}. In order to give a strict definition and 
	recast system (\ref{s1}) in a functional frame, we introduce the metric spaces
\begin{align*}
\mathcal Y:=D(\mathcal L)\times D(\mathcal L)\times\mathcal U\quad\text{and}\quad \mathcal Z:=L^2(\Omega)\times   
L^2(\Omega)\times L^\infty(\Omega),
\end{align*}
endowed with the following metrics. For $\psi_i=(y_i,p_i,u_i) \in {\mathcal Y}$ and 
$\zeta_i=(\xi_i,\eta_i,\rho_i) \in {\mathcal Z}$, $i\in\{1,2\}$,
\begin{align*}
& d_{\mathcal Y}(\psi_1,\psi_2):=|y_1-y_2|_{L^2(\Omega)}+|p_1-p_2|_{L^2(\Omega)}+|u_1-u_2|_{L^1(\Omega)},\\
& d_{\mathcal Z}(\zeta_1,\zeta_2):=|\xi_1-\xi_2|_{L^2(\Omega)}+|\eta_1-\eta_2|_{L^2(\Omega)}+
|\rho_1-\rho_2|_{L^\infty(\Omega)}.
\end{align*}
Both metrics are shift-invariant. We denote by $\mathbb B_{\mathcal Y}(\psi;\alpha)$ the closed ball in ${\mathcal Y}$, centered 
at $\psi$ and with radius $\alpha$. The notation for the ball $\mathbb B_{\mathcal Z}(\zeta;\alpha)$ is analogous.
{Then the optimality mapping is defined as the set-valued mapping $\Phi:\mathcal Y\twoheadrightarrow\mathcal Z$ } given by
\begin{align}\label{optmapping}
\Phi(y,p,u) =\left( \begin{array}{c}
\mathcal Ly - f(\cdot,y,u) \\
\mathcal Lp- H_y(\cdot,y,p,u) \\
H_u(\cdot,y,p,u)+ N_{\mathcal U}(u)
\end{array} \right) .
\end{align}
{Then the optimality system (\ref{s1}) can be recast as the inclusion }
\begin{align}\label{inc}
0\in\Phi(y,p,u).
\end{align}
Our purpose is to study the stability of system (\ref{s1}), or equivalently of inclusion (\ref{inc}), {with respect to perturbations on 
	the left-hand side}. 
{From now on, we denote $\bar \psi := (\bar y, \bar p, \bar u) = (y_{\bar u},p_{\bar u},\bar u)$ , where $\bar u$ is the fixed local solution of problem (\ref{cost})--(\ref{system}).}

\begin{dfntn}\label{Dsmsr}
	{The} optimality mapping $\Phi:\mathcal Y\twoheadrightarrow \mathcal Z$ is {called} strongly 
	{H\"older subregular} with exponent $\lambda>0$ at $(\bar\psi,0)$ if there exist positive numbers 
	$\alpha_1,\alpha_2$ {and $\kappa$} such that 
	\begin{align}\label{Essr}
	d_{\mathcal Y}(\psi,\bar\psi)\le\kappa d_{\mathcal Z}(\zeta,0)^\lambda
	\end{align}
	for all $\psi\in \mathbb B_{\mathcal Y}(\bar \psi;\alpha_1)$ and $\zeta\in \mathbb B_{\mathcal Z}(0;\alpha_2)$ 
	satisfying $\zeta\in\Phi(\psi)$. 
\end{dfntn} 
{More explicitly, the inequality (\ref{Essr}) reads as 
	\begin{align} \label{EHe}
	|y-y_{\bar u}|_{L^2(\Omega)}+|p-p_{\bar u}|_{L^2(\Omega)}+|u-\bar u|_{L^1(\Omega)}\le
	\kappa\Big( |\xi|_{L^2(\Omega)}+|\eta|_{L^2(\Omega)}+|\rho|_{L^\infty(\Omega)}\Big)^\lambda.
	\end{align}}

Hence, if the optimality mapping is {strongly H\"older subregular}, all solutions of the system 
\begin{align}\label{s1per}
\left\{ \begin{array}{cll}
\xi&=&\mathcal Ly-f(\cdot,y,u),\\
\eta&=&\mathcal Lp-H_{y}(\cdot,y,p,u),\\
\rho&\in& H_{u}(\cdot,y,p) + N_{\mathcal U}(u).
\end{array} \right.
\end{align}
that are near  $(y_{\bar u},p_{\bar u},\bar u)$  satisfy the H\"older estimate {(\ref{EHe})} 
with respect to the perturbations {$\zeta = (\xi,\eta,\rho)$}, provided they are small enough. The subregularity property is weaker than the well known strong regularity (see \cite[pp. 178-179]{Dontchevbook}); this allows to relax the assumptions to prove stability.

\begin{rmrk}\label{strictlocal}
	If 
	$\Phi$ is strongly H\"older subregular at $(\bar \psi, 0)$, then from (\ref{Essr}) applied with $\zeta = 0$ we obtain that $\bar \psi$ is the unique solution of (\ref{inc})
	in $B_{\mathcal Y}(\bar \psi;\alpha_1)$, hence $\bar u$ is the unique local solution of problem (\ref{cost})--(\ref{system})
	in this ball. In particular, $\bar u$ is a strict local minimizer.
\end{rmrk}

We are now ready to state our main result.

\begin{thrm}\label{Ssr} 
	{Let Assumption \ref{A2} hold. Then the optimality mapping $\Phi$ is strongly H\"older subregular 
		at $(\bar\psi,0)$ with exponent $\lambda = 1/k^*$.}
\end{thrm}

\begin{proof}
	Let $\alpha$ and  $c$ be the positive numbers in Lemma \ref{Esslem}. 
	Let $\zeta=(\xi,\eta,\rho)\in\mathcal \mathbb B_{\mathcal Z}(0;1)$ and 
	$\psi=(y,p,u)\in \mathbb B_{\mathcal Y}(\bar\psi;\alpha)$ such that  $\zeta\in\Phi(\psi)$. 
	By a standard argument, we can find $c_1>0$ (independent of $\psi$ and $\zeta$) such that
	\begin{align}\label{mt1}
	|y-y_{u}|_{L^\infty(\Omega)}+|p-p_{u}|_{L^\infty(\Omega)}\le c_1\Big(|\xi|_{L^2(\Omega)}+|\eta|_{L^2(\Omega)}\Big).
	\end{align}
	Since $H_u$ is locally Lipschitz uniformly in the first variable, and the sets $\{y_u:u\in\mathcal U\}$, $\{p_u:u\in\mathcal U\}$ 
	are bounded in $C(\bar\Omega)$, there exists $c_2>0$ (independent of $\psi$) such that 
	\begin{align}\label{mt2}
	|H_{u}(\cdot,y,p)-H_{u}(\cdot,y_u,p_u)|_{L^\infty(\Omega)}\le c_2\Big
	(	|y-y_{u}|_{L^\infty(\Omega)}+|p-p_{u}|_{L^\infty(\Omega)}\Big)
	\end{align}
	Define $\nu:=\rho+H_{u}(\cdot,y_u,p_u) - H_{u}(\cdot,y,p).$ By (\ref{mt1}) and (\ref{mt2}), 
	there exists $c_3>0$ (independent of $\psi$ and $\zeta$) such that 
	\begin{align*}
	|\nu|_{L^\infty(\Omega)}\le c_3\Big(|\xi|_{L^2(\Omega)}+|\eta|_{L^2(\Omega)}+
	|\rho|_{L^\infty(\Omega)}\Big)=c_3|\zeta|_{\mathcal Z}.
	\end{align*}
	As $\rho\in H_{u}(\cdot,y,p)+ N_{\mathcal U}(u)$, we have $\nu\in H_{u}(\cdot, y_u,p_u)+ N_{\mathcal U}(u)$. 
	Then by Lemma \ref{Esslem},  
	\begin{align}\label{mt3}
	|u-\bar u|_{L^1(\Omega)}\le c|\nu|_{L^\infty(\Omega)}^{\frac{1}{{k^*}}}\le 
	cc_3^{\frac{1}{{k^*}}}|\zeta|^{\frac{1}{{k^*}}}_{\mathcal Z}:=c_4|\zeta|^{\frac{1}{{k^*}}}_{\mathcal Z}.
	\end{align}
	Now, by Proposition \ref{contstatel1}, there exists $c_5>0$ (independent of $\psi$) such that 
	$|y_{u}-y_{\bar u}|_{L^2(\Omega)}\le c_5|u-\bar u|_{L^1(\Omega)}$. Consequently, by (\ref{mt3})
	\begin{align*}
	|y-y_{\bar u}|_{L^2(\Omega)}&\le |y-y_{ u}|_{L^2(\Omega)}+|y_{u}-y_{\bar u}|_{L^2(\Omega)}\\
	&\le c_1\text{meas}\hspace{0.08cm}\Omega^{\frac{1}{2}}\Big(|\xi|_{L^2(\Omega)}+|\eta|_{L^2(\Omega)}\Big)+
	c_5|u-\bar u|_{L^1(\Omega)}\\
	&\le (c_1\text{meas}\hspace{0.08cm}\Omega^{\frac{1}{2}}+c_5c_4)|\zeta|^{\frac{1}{{k^*}}}_{\mathcal Z}
	=:c_6|\zeta|^{\frac{1}{{k^*}}}_{\mathcal Z}.
	\end{align*}
	Analogously, there exists $c_7>0$ (independent of $\psi$ and $\zeta$) such that
	\begin{align*}
	|p-p_{\bar u}|_{L^2(\Omega)}\le c_7 |\zeta|^{\frac{1}{{k^*}}}_{\mathcal Z}.
	\end{align*}
	Putting all together, 
	\begin{align*}
	|y-y_{\bar u}|_{L^2(\Omega)}+|p-p_{\bar u}|_{L^2(\Omega)}+|u-\bar u|_{L^1(\Omega)}\le 
	(c_4+c_6+c_7)|\zeta|^{\frac{1}{{k^*}}}_{\mathcal Z}.
	\end{align*}
	Finally, let $\alpha_1:=\alpha$, $\alpha_2:=1$ and $\kappa:=c_4+c_6+c_7$. 
	Since the constants $c_4, c_6$ and $c_7$ are independent of $\psi$ and $\zeta$, so is $\kappa$. 
	Thus we have (\ref{Essr})
	for all $\psi\in \mathbb B_{\mathcal Y}(\bar \psi;\alpha_1)$ and $\zeta\in 
	\mathbb B_{\mathcal Z}(0;\alpha_2)$ satisfying $\zeta\in\Phi(\psi)$. 
\end{proof}

{The strong subregularity property defined above does not require existence of solutions of the perturbed 
	inclusion (\ref{s1per}) in a neighborhood of the reference solution $\bar \psi$. 
	The next theorem answers the existence question.} 

\begin{thrm}\label{Ssrexist}
	Let Assumption \ref{A2} hold. For each $\varepsilon>0$ there exists $\delta>0$ such that for every 
	$\zeta\in \mathbb B_{\mathcal Z}(0;\delta)$  there exists $\psi\in \mathbb B_{\mathcal Y}(\bar \psi;\varepsilon)$ 
	satisfying the inclusion $\zeta\in \Phi(\psi)$.
\end{thrm}

\begin{proof}
	For each $u\in\mathcal U$ and $\zeta=(\xi,\eta,\rho)\in \mathcal Z$, define 
	$\nu_{u,\zeta}:=\rho+H_{u}(\cdot,y_u,p_u)-H_{u}(\cdot,y_{u,\zeta},p_{u,\zeta})$, 
	where $y_{u,\zeta}$ and $p_{u,\zeta}$ are the unique solutions of
	\begin{align}\label{s1pers1}
	\left\{ \begin{array}{cll}
	\mathcal Ly&=&f(\cdot,y,u)+\xi,\\
	\mathcal Lp&=&H_{y}(\cdot,y,p,u)+\eta.
	\end{array} \right.
	\end{align}
	By a standard argument, one can find positive numbers $c_1$ and $c_2$ such that 
	\begin{align}\label{dl2}
	|y_{u,\zeta}-y_{u}|_{L^2(\Omega)}+|p_{u,\zeta}-p_{u}|_{L^2(\Omega)}\le 
	c_1\Big(|\xi|_{L^2(\Omega)}+|\eta|_{L^2(\Omega)}\Big),
	\end{align}
	and $|\nu_{u,\zeta}|_{L^\infty(\Omega)}\le c_2|\zeta|_{\mathcal Z}$ for all $u\in\mathcal U$ 
	and $\zeta\in\mathcal Z$. Let $\varepsilon>0$ be arbitrary. 
	By Lemma \ref{exisrho}, the exists $\delta_0>0$ such that for each $\nu\in \mathbb B_{L^\infty}(0;\delta_0)$ 
	there exists $u\in \mathcal U\cap \mathbb B_{L^1}(\bar u;\varepsilon/2)$ satisfying $\nu\in\sigma_u+ N_{\mathcal U}(u)$. 
	Define $\delta:=\min\{c_2^{-1}\delta_0,(2c_1)^{-1}\varepsilon\}$ and let $\zeta^*\in \mathbb B_{\mathcal Z}(0;\delta)$ 
	be arbitrary; we will prove that there exists $u^*\in\mathcal U\cap \mathbb B_{L^1}(\bar u;\varepsilon/2)$ 
	such that $\nu_{u^*,\zeta^*}\in\sigma_{u^*}+N_{\mathcal U}(u^*)$. First, observe that 
	\begin{align*}
	|\nu_{u,\zeta^*}|_{L^\infty(\Omega)}\le c_2|\zeta^*|_{\mathcal Z}\le \delta_0\quad\forall u\in\mathcal U.
	\end{align*}
	Therefore, by Lemma \ref{exisrho}, we can {inductively define} 
	a sequence $\{u_{k}\}_{k=1}^\infty\subset \mathcal U$ 
	such that $\nu_{u_{k},\zeta^*}\in\sigma_{u_{k+1}}+N_{\mathcal U}(u_{k+1})$ and 
	$|u_k-\bar u|_{L^1(\Omega)}\le \varepsilon/2$ for all $k\in\mathbb N$. Since $\mathcal U$ is weakly 
	compact in $L^2(\Omega)$, we may assume that $u_{k}\rightharpoonup u^*$ 
	weakly in $L^2(\Omega)$ for some $u^*\in\mathcal U$.  Weak convergence in $L^2(\Omega)$ 
	implies weak convergence in $L^1(\Omega)$ and $\mathbb B_{L^1}(\bar u;\varepsilon/2)$ is weakly sequentially 
	closed in $L^1(\Omega)$, therefore $u^*\in \mathbb B_{L^1}(\bar u; \varepsilon/2)$. Using Proposition \ref{weakconvergence}, one can see that 
	$\nu_{u_k,\zeta^*}-\sigma_{u_{k+1}}\to\nu_{u^*,\zeta^*}-\sigma_{u^*}$ in $L^\infty(\Omega)$,
	and consequently that $\nu_{u^*,\zeta^*}\in\sigma_{u^*}+N_{\mathcal U}(u^*)$. 
	We conclude then that $\zeta^*\in\Phi(\psi^*)$, where $\psi^*:=(y_{u^*,\zeta^*}, p_{u^*,\zeta^*},u^*)$. 
	Finally, by definition of $\delta$ and (\ref{dl2}) 
	\begin{align*}
	|\psi^*-\bar\psi|_{\mathcal Y}\le c_1|\zeta|_{\mathcal Z}+\varepsilon/2\le\varepsilon.
	\end{align*}
	Thus, $\zeta^*\in\Phi(\psi^*)$ and $\psi^*\in \mathbb B_{\mathcal Y}(\bar\psi;\varepsilon)${, which completes the proof}. 
\end{proof}

{The next theorem claims that \emph{all} solutions of the perturbed optimality system (\ref{s1per}) are arbitrarily close to the solution 
	of the unperturbed optimality system, provided that the solution of the latter is globally unique,  
	Assumption \ref{A2} holds, and the perturbation is sufficiently small.}

\begin{thrm}\label{Ssrexisuni}
	Let Assumption \ref{A2} hold and suppose additionally that $\bar\psi$ is the unique element 
	of $\mathcal Y$ that satisfies $0\in\Phi(\bar\psi)$. For each $\varepsilon>0$ there exists $\delta>0$ 
	such that if $\zeta\in \mathbb B_{\mathcal Z}(0;\delta)$ and $\psi\in\mathcal Y$ satisfy $\zeta\in \Phi(\psi)$,  
	then $\psi\in \mathbb B_{\mathcal Y}(\bar \psi;\varepsilon)$.
\end{thrm}

\begin{proof}
	Let $\delta_0$ and $c_0$ be the positive numbers in Lemma \ref{Esslemuni}. 
	Let $\zeta=(\xi,\eta,\rho)\in{\mathcal Z}$ and $\psi=(y,p,u)\in\mathcal Y$ {be} such that  $\zeta\in\Phi(\psi)$. 
	Define $\nu:=\rho+H_{u}(\cdot,y_u,p_u) - H_{u}(\cdot,y,p).$ Arguing as in the proof of Theorem \ref{Ssr}, 
	we can find positive numbers $c_1$ and $c_2$ (independent of $\psi$ and $\zeta$) such that 
	$|\nu|_{L^\infty(\Omega)}\le c_1|\zeta|_{\mathcal Z}$ and 
	\begin{align*}
	|y-y_{\bar u}|_{L^2(\Omega)}+|p-p_{\bar u}|_{L^2(\Omega)}\le 
	c_2\Big(|\zeta|_{\mathcal Z}+|u-\bar u|_{L^1(\Omega)}\Big).
	\end{align*}
	Let $\delta:=\min\{c_1^{-1}\delta_0,(2c_0c_2)^{-k^*}c_1^{-1}\varepsilon^{k^*},(2c_2)^{-1}\varepsilon\}$ 
	and suppose that $\zeta\in \mathbb B_{\mathcal Z}(0;\delta)$.
	As $\rho\in H_{u}(\cdot,y,p)+ N_{\mathcal U}(u)$, we have $\nu\in H_{u}(\cdot, y_u,p_u)+ N_{\mathcal U}(u)$. 
	By Lemma \ref{Esslemuni},   
	\begin{align*}
	|u-\bar u|_{L^1(\Omega)}\le c_0|\nu|_{L^\infty(\Omega)}^{\frac{1}{{k^*}}}\le 
	c_0c_1^\frac{1}{{k^*}}|\zeta|^{\frac{1}{{k^*}}}_{\mathcal Z}\le c^{-1}_2\varepsilon/2.
	\end{align*}
	Thus,
	\begin{align*}
	|y-y_{\bar u}|_{L^2(\Omega)}+|p-p_{\bar u}|_{L^2(\Omega)}+|u-\bar u|_{L^1(\Omega)}\le    
	c_2\Big(\delta+c^{-1}_2\varepsilon/2\Big)\le\varepsilon.
	\end{align*}
\end{proof}

\section{Nonlinear Perturbations}\label{Section nonlin}
In this section we apply the subregularity results in Section \ref{Section Stab} for studying the effect of certain nonlinear perturbations on the optimal solution.
We consider the following family of problems 
\begin{align}\label{costpp}
\quad\min_{u\in\mathcal U}\left\lbrace\int_{\Omega}\Big[g(x,y,u)+\eta(x,y,u)\Big]\,dx\right\rbrace,
\end{align}
subject to 
\begin{align}\label{systempp}
\left\{ \begin{array}{cclcc}
-\dive\big(A(x)\nabla y\big)+d(x,y)+\xi(x,y)&=&\beta(x)u& \text{in}& \Omega  \\
\\ A(x)\nabla y\cdot \nu+b(x)y&=&0 &\text{on}& \partial\Omega.
\end{array} \right.
\end{align}

In order to specify the perturbations $\xi$ and $\eta$ under consideration and their topology, we begin the section recalling some elementary notions of functional analysis.

As usual,  $C(\mathbb R^s)$ denotes the space of all continuous functions $\omega:\mathbb R^s\to\mathbb R$. For each $m\in\mathbb N$, let $K_m$ denote the closed ball in $\mathbb R^s$ centered at zero with radius $m$. Consider the metric on $C(\mathbb R^s)$ given by
\begin{align*}
d_{C}(\omega_1,\omega_2):=\sum_{m=1}^{\infty}\frac{1}{2^m}\frac{|\omega_1-\omega_2|_{L^\infty(K_m)}}{1+|\omega_1-\omega_2|_{L^\infty(K_m)}}.
\end{align*}
This metric induces the compact-convergence topology on $C(\mathbb R^s)$. In this topology, a sequence $\{\omega_m\}_{m=1}^\infty\subset C(\mathbb R^s)$ converges to $\omega\in C(\mathbb R^s)$ if and only if 
$|\omega-\omega_m|_{L^\infty(K)}\to0$ for every compact set $K\subset\mathbb R^s$. This topology is also known as the compact-open topology, see \cite[Chapter 7]{Munkres}. The following lemma is straightforward and follows from the definition of $d_{C}$.
\begin{lmm}\label{metlem}
	For each compact set $K\subset\mathbb R^s$ there exists $m\in\mathbb N$ such that 
	\begin{align*}
	|\omega_1-\omega_2|_{L^\infty(K)}\le 2^{m}d_{C}(\omega_1,\omega_2)
	\end{align*}
	for all $\omega_1,\omega_2\in C(\mathbb R^s)$ such that $d_{C}(\omega_1,\omega_2)\le {2^{{-m}}}$.
\end{lmm}
\begin{proof}
	Let $K$ be a compact subset of $\mathbb R^s$. There exists $i\in\mathbb N$ such that $K\subset K_i$, where $K_i$ denotes the closed ball in $\mathbb R^s$ centered at zero with radius $i$. Now, by definition of the metric $d_{C}$,
	\begin{align*}
		\frac{|\omega|_{L^\infty(K_i)}}{1+|\omega|_{L^\infty(K_i)}}\le 2^i d_{C}(\omega,0)\quad\forall \omega\in C(\mathbb R^s).
	\end{align*}
	Hence,
	\begin{align*}
		|\omega|_{L^\infty(K_i)}&\le \frac{2^{i}d_{C}(\omega,0)}{1-2^{i}d_{C}(\omega,0)}\le 2^{i+1}d_{C}(\omega,0)
	\end{align*}
	for all $\omega\in C(\mathbb R^s)$ with $d_{C}(\omega,0)\le 2^{-(i+1)}$. Let $m=i+1$. Then
	\begin{align*}
	|\omega_2-\omega_1|_{L^\infty(K)}\le|\omega_2-\omega_1|_{L^\infty(K_i)}\le 2^m d_{C}(\omega_2-\omega_1,0)=2^m d_{C}(\omega_2,\omega_1)
	\end{align*}
	for all $\omega_1,\omega_2\in C(\mathbb R^s)$ with $d_{C}(\omega_1,\omega_2)\le {2^{{-m}}}$.
\end{proof}

\subsection{The perturbations}
We begin describing the space of perturbations appearing in equation (\ref{systempp}). Let $\Upsilon_{s}$ be the set of all continuously differentiable functions $\xi: \mathbb R^n\times \mathbb R\to\mathbb R$ such that   $	d_y(x,y)+\xi_{y}(x,y)\ge 0$
for all $x\in\Omega$ and $y\in\mathbb R$. The set $\Upsilon_{s}$ does not constitute a linear space, but it allows to have well-defined states for each perturbation.
\begin{prpstn}\label{contstateper}
	For each $u\in\mathcal U$ and $\xi\in \Upsilon_{s}$ there exists a unique function $y^{\xi}_u\in D(\mathcal L)$ satisfying 
	\begin{align*}
	\mathcal Ly^\xi_u+d(\cdot,y^\xi_u)+\xi(\cdot,y^\xi_u)=\beta(\cdot)u.
	\end{align*}
	Moreover, there exist positive numbers $M$ and $\delta$ such that $|y_{u}^\xi|_{L^\infty(\Omega)}\le M$
	for all $u\in\mathcal U$ and $\xi\in\Upsilon_{s}$ with $d_{C}(\xi,0)\le\delta$.
\end{prpstn}
\begin{proof}
	The existence follows from \cite[Theorem 4.8]{TroeltzschPde}. Moreover, also from this theorem, there exists $c>0$ such that
	\begin{align*}
	|y^\xi_u|_{L^\infty(\Omega)}\le c\big|\beta(\cdot)u-d(\cdot,0)-\xi(\cdot,0)\big|_{L^\infty(\Omega)}
	\end{align*}
	for all $u\in\mathcal U$ and $\xi\in\Upsilon_{s}$. Let $K:=\bar\Omega\times\{0\}$, then by Lemma \ref{metlem} there exists $m\in\mathbb N$ such that
	\begin{align*}
	|y^\xi_u|_{L^\infty(\Omega)}&\le c\Big( |\beta|_{L^\infty(\Omega)}|u|_{L^\infty(\Omega)}+|d(\cdot,0)|_{L^\infty(\Omega)}+|\xi|_{L^\infty(K)}\Big)\\
	&\le c\Big( |\beta|_{L^\infty(\Omega)}\sup_{u \in \mathcal U}|u|_{L^\infty(\Omega)}+|d(\cdot,0)|_{L^\infty(\Omega)}+2^{m}d_{C}(\xi,0)\Big)\\
	& \le c\Big( |\beta|_{L^\infty(\Omega)}\sup_{u \in \mathcal U}|u|_{L^\infty(\Omega)}+|d(\cdot,0)|_{L^\infty(\Omega)}+1\Big)
	\end{align*}
	for all $u\in\mathcal U$ and $\xi\in\Upsilon_{s}$ with $d_{C}(\xi,0)\le 2^{{-m}}$. The result follows defining $\delta:=2^{{-m}}$ and 
	\begin{align*}
	M:=c\Big( |\beta|_{L^\infty(\Omega)}\sup_{u \in \mathcal U}|u|_{L^\infty(\Omega)}+|d(\cdot,0)|_{L^\infty(\Omega)}+1\Big).
	\end{align*}
\end{proof}

We now proceed to describe the perturbations appearing in the cost functional (\ref{costpp}). Consider the set $\Upsilon_{c}$ of all continuously differentiable functions $\eta:\mathbb R^n\times\mathbb R\times\mathbb R\to\mathbb R$ such that $\eta(x,y,\cdot)$ is convex for all $x\in\Omega$ and $y\in\mathbb R$. We have the following result concerning the adjoint variable of the perturbed problem. Its proof is similar to the one of Proposition \ref{contstateper}.

\begin{prpstn}\label{contcostateper}
	For each $u\in\mathcal U$, $\xi\in \Upsilon_{s}$ and $\eta\in\Upsilon_c$ there exists a unique function $p^{\xi,\eta}_u\in D(\mathcal L)$ satisfying 
	\begin{align*}
	\mathcal Lp^{\xi,\eta}_u+\big[d_y(\cdot,y_u^\xi)+\xi_y(\cdot,y^\xi_u)\big]p^{\xi,\eta}_u=g_y(\cdot,y_{u}^\xi,u)+\eta_y(\cdot,y_{u}^\xi,u).
	\end{align*}
	Moreover, there exist positive numbers $M$ and $\delta$ such that $|p_{u}^{\xi,\eta}|_{L^\infty(\Omega)}\le M$
	for all $u\in\mathcal U$, $\xi\in\Upsilon_{s}$  and $\eta\in\Upsilon_{c}$ with $d_{C}(\xi,0)+d_{C}(\xi_y,0)+d_{C}(\eta_y,0)\le\delta$.
\end{prpstn}

We denote $\Upsilon:=\Upsilon_{s}\times\Upsilon_{c}$, and write  $\zeta:=(\xi,\eta)$ for a generic element of $\Upsilon$. We endow $\Upsilon$ with the pseudometric $d_\Upsilon:\Upsilon\times\Upsilon\to[0,\infty)$ given by
\begin{align*}
d_{\Upsilon}(\zeta,\zeta'):=d_{C}(\xi,\xi')+d_{C}(\xi_y,\xi_y')+d_{C}(\eta_y,\eta_y')+d_{C}(\eta_u,\eta_u').
\end{align*}

\subsection{The stability result}

We are  now ready to state problem (\ref{costpp})-(\ref{systempp}) in a precise way. Given $\zeta\in\Upsilon$, problem $\mathcal P_\zeta$ is given by
\begin{align}\label{Perprob}
\quad\min_{u\in\mathcal U}\left\lbrace\mathcal J_\zeta(u):=\int_{\Omega}\Big[g(x,y_u^\xi,u)+\eta(x,y_u^\xi,u)\Big]\,dx\right\rbrace.
\end{align}
Due to the convexity of the cost in the control variable, each  problem  $\mathcal P_\zeta$ has at least one global solution. For each $\zeta\in\Upsilon$, we fix a local minimizer $\hat u_{\zeta}\in\mathcal U$ of problem 
$\mathcal P_{\zeta}$.  By the local minimum principle, for each $\zeta=(\xi,\eta)\in\Upsilon$, the triple $(\hat y_\zeta,\hat p_\zeta,\hat u_\zeta):=({y_{\hat u_\zeta}^\xi},{p_{\hat u_\zeta}^{\xi,\eta}},\hat u_{\zeta})$ satisfies the system
\begin{align}\label{s1perper}
\left\{ \begin{array}{cll}
0&=&\mathcal Ly-f(\cdot,y,u)-\xi(\cdot,y),\\
0&=&\mathcal Lp-H_{y}(\cdot,y,p,u)+\eta_y(\cdot,y,u)-\xi_y(\cdot,y)p,\\
0&\in& H_{u}(\cdot,y,p)+ \eta_u(\cdot,y,u) + N_{\mathcal U}(u).
\end{array} \right.
\end{align}

As  a consequence of Theorem \ref{Ssr}, we have the following result.

\begin{thrm}\label{Thmreg0}
	Let Assumption \ref{A2} hold. There exist positive numbers $\alpha,\alpha'$ and $c$ such that
	\begin{align*}
	|\hat y_{\zeta}-y_{\bar u}|_{L^2(\Omega)}+|\hat p_{\zeta}-p_{\bar u}|_{L^2(\Omega)}+|\hat u_\zeta-\bar u|_{L^1(\Omega)}\le c d_\Upsilon(\zeta,0)^{1/k^*}
	\end{align*}
	for all $\zeta\in \Upsilon$ such that $|\hat u_{\zeta}-\bar u|_{L^1(\Omega)}\le\alpha$ and $d_{\Upsilon}(\zeta,0)\le\alpha'$.
\end{thrm}
\begin{proof}
	By Theorem \ref{Ssr}, the mapping $\Phi$ is  strongly H\"older subregular at $(\bar\psi,0)$ with exponent $1/k^*$. 
	Let $\alpha_1,\alpha_2$ and $\kappa$ be the positive numbers in the definition of strong subregularity.
	By Proposition \ref{contstateper} and \ref{contcostateper} there exist positive numbers $M$ and $\delta_0$ such that 
	\begin{align*}
	|y_u^\xi|_{L^\infty(\Omega)}+|p^{\xi,\eta}_u|_{L^\infty(\Omega)}\le M
	\end{align*}
	for all $u\in\mathcal U$ and  $\zeta\in\Upsilon$ with $d_{\Upsilon}(\zeta,0)\le\delta_0$. Let $K:=\bar\Omega\times[-M,M]$. By Lemma \ref{metlem}, there exists $m\in\mathbb N$ such that 
	\begin{align*}
	|\xi(\cdot, y_{u}^\xi)|_{L^2(\Omega)}\le \text{meas}\hspace*{0.06cm}\Omega^{\frac{1}{2}} |\xi|_{L^\infty(K)}\le 2^{m} \text{meas}\hspace*{0.06cm}\Omega^{\frac{1}{2}}d_{C}(\xi,0)\le 2^m  \text{meas}\hspace*{0.06cm}\Omega^{\frac{1}{2}} d_{\Upsilon}(\zeta,0)
	\end{align*}
	for all $u\in\mathcal U$ and $\zeta\in\Upsilon$ with $d_{\Upsilon}(\zeta,0)\le \min\{2^{{-m}},\delta_0\}$.  Repeating this argument, we can find positive numbers $\delta$ and $c_0$ such that 
	\begin{align}\label{oflin}
	|\xi(\cdot, y_{u}^\xi)|_{L^2(\Omega)}+|\xi_{y}(\cdot,y^{\xi}_u)p_{u}^{\xi,\eta}|_{L^2(\Omega)}+|\eta_y(\cdot,y_{u}^\xi,u)|_{L^2{\Omega}}+|\eta_{u}(\cdot,y_{u}^\xi,u)|_{L^\infty}\le c_0 d_{\Upsilon}(\zeta,0)
	\end{align}
	for all $u\in\mathcal U$ and $\zeta\in\Upsilon$ with $d_{\Upsilon}(\zeta,0)\le \delta$. Using Proposition \ref{contstatel1} and Lemma \ref{metlem}, one can find positive numbers $\alpha$ and $\delta'$ such that 
	\begin{align*}
	|\hat y_{\zeta}-y_{\bar u}|_{L^2(\Omega)}+|\hat p_{\zeta}-p_{\bar u}|_{L^2(\Omega)}+|\hat u_\zeta-\bar u|_{L^1(\Omega)}\le \alpha_1
	\end{align*}
	for all $\zeta\in \Upsilon$ with $|\hat u_\zeta-\bar u|_{L^1(\Omega)}\le\alpha$ and $d_{\Upsilon}(\zeta,0)\le\delta'$. Observe that by (\ref{s1perper}), we have
	\begin{align*}
	\left( \begin{array}{c}
	\xi(\cdot,\hat y_\zeta) \\
	-\eta_y(\cdot,\hat y_\zeta,\hat u_\zeta)+\xi_y(\cdot,\hat y_\zeta)\hat p_\zeta \\
	-\eta_u(\cdot,\hat y_\zeta,\hat u_\zeta)
	\end{array} \right) \in \Phi(\hat y_\zeta,\hat p_\zeta,\hat u_{\zeta})
	\end{align*}	
	for all $\zeta\in\Upsilon$. Let $\alpha':=\min\{c_0^{-1}\alpha_2,\delta,\delta'\}$. Then by H\"older subregularity of $\Phi$ and (\ref{oflin}),
	\begin{align*}
	|\hat y_{\zeta}-y_{\bar u}|_{L^2(\Omega)}+|\hat p_{\zeta}-p_{\bar u}|_{L^2(\Omega)}+|\hat u_\zeta-\bar u|_{L^1(\Omega)}\le \kappa c_0^{\frac{1}{k^*}} d_{\Upsilon}(\zeta,0)^{\frac{1}{k^*}}
	\end{align*}
	for all $\zeta\in\Upsilon$ such that $|\hat u_\zeta-\bar u|_{L^1(\Omega)}\le\alpha$ and $d_{\Upsilon}(\zeta,0)\le \alpha'$. The result follows defining $c:=\kappa c_0^{\frac{1}{k^*}}$.
\end{proof}

\subsection{An application: Tikhonov regularization}
In what follows we present an application of the theory derived in the previous chapters, namely the so-called Tikhonov regularization. 
For a more detailed description and an account of the state of art, the reader 
is referred to \cite{MR3810878,MR2986517,MR3780469}. We derive estimates on the convergence rate of the solution of the regularized problem when the regularization parameter tends to zero. 
The results that appear in the literature require the so-called structural assumption and positive-definiteness 
(in some sense) of the  second derivative of the objective functional. Using Theorem \ref{Ssr}, we can obtain 
this results under weaker assumptions than used in the literature so far. One can compare this results with 
\cite[Theorem 4.4]{MR3810878} (where a tracking problem with semilinear elliptic equation  is considered) 
when it comes to stability of the controls. In Section \ref{Section5}, we give more details on how the assumptions 
in the literature interplay with Assumption \ref{A2}.

We consider the following family of problems $\{\mathcal P_{\varepsilon}\}_{\varepsilon\ge0}$. 
\begin{align}\label{costp}
\quad\min_{u\in\mathcal U}\left\lbrace\int_{\Omega}g(x,y,u)\,dx+\frac{\varepsilon}{2}\int_{\Omega} u^2\,dx\right\rbrace,
\end{align}
subject to 
\begin{align}\label{systemp}
\left\{ \begin{array}{cclcc}
-\dive\big(A(x)\nabla y\big)+d(x,y)&=&\beta(x)u& \text{in}& \Omega  \\
\\ A(x)\nabla y\cdot \nu+b(x)y&=&0 &\text{on}& \partial\Omega.
\end{array} \right.
\end{align}

\begin{lmm}\label{lemfin}
		Let Assumption \ref{A2} be fulfilled. For every $\alpha>0$ there exists $\varepsilon_\alpha>0$ such that for every $\varepsilon\in(0,\varepsilon_\alpha)$ problem $\mathcal P_{\varepsilon}$ has a local solution  $\hat u_{\varepsilon}\in\mathcal U\cap \mathbb B_{L^1}(\bar u;\alpha)$.
\end{lmm}
\begin{proof}
	Let $\alpha>0$ be arbitrary. By Remark \ref{strictlocal}, $\bar u$ is a strict local minimizer, hence there exists $\alpha^*\le \alpha$ such that $\mathcal J(\bar u)<\mathcal J(u)$ for all $\bar u\neq u\in\mathcal U\cap \mathbb B_{L^1(\Omega)}(\bar u;\alpha^*)$.
	Consider the family of problems $\mathcal P^*_{\varepsilon}$ given by
	\begin{align}\label{label2}
		\displaystyle\min_{\mathcal U\cap \mathbb B_{L^1}(\bar u;\alpha^*)}\left\lbrace \mathcal J(u)+\frac{\varepsilon}{2}|u|_{L^2(\Omega)}^2\right\rbrace .
	\end{align}
	Each problem  $\mathcal P^*_{\varepsilon}$ has a global solution $\hat u_\varepsilon$. There exists $\varepsilon^*>0$ 
	such that $|\hat u_{\varepsilon}-\bar u|_{L^1(\Omega)}\le \alpha^*/2$ for all $\varepsilon\in(0,\varepsilon^*)$.  	Suppose the opposite. Then there exists a sequence $\{\varepsilon_{k}\}_{k=1}^\infty$ converging to zero such that 
	$|\hat u_{\varepsilon_k}-\bar u|_{L^1(\Omega)}>\alpha^*/2$ for all $k\in\mathbb N$. Since $\mathcal U\cap \mathbb B_{L^1}(\bar u;\alpha^*)$ is weakly 
	compact in $L^2(\Omega)$, we may assume without loss of generality that $u_{\varepsilon_{k}}\rightharpoonup u^*$ for some 
	$u^*\in\mathcal U\cap \mathbb B_{L^1}(\bar u;\alpha^*)$. Since $y_{\hat u_{\varepsilon_k}}\to y_{u^*}$ in $C(\bar\Omega)$, we obtain that
	\begin{align*}
	\mathcal J(u^*) \leq \lmi_{k\to\infty}\Big[\mathcal J(\hat u_{\varepsilon_k})+\frac{\varepsilon_{k}}{2}|\hat u_{\varepsilon_{k}}|_{L^2(\Omega)}^2\Big]
	\le \lmi_{k\to\infty}\Big[\mathcal J(\bar u)+\frac{\varepsilon_{k}}{2}|\bar u|_{L^2(\Omega)}^2\Big]=\mathcal J(\bar u).
	\end{align*}
	Therefore $u^*=\bar u$ since $u^*\in\mathcal U\cap\mathbb B_{L^1(\Omega)}(\bar u;\alpha^*)$ and $\bar u$ is strict local minimum. 
	By Proposition \ref{bangbang}, $u^*=\bar u$ is bang-bang. Weak convergence in $L^2(\Omega)$ 
	implies weak convergence in $L^1(\Omega)$; consequently, by Lemma \ref{weakimpliesstrong}, 
	$\hat u_{\varepsilon_{k}}\to u^*$ in $L^1(\Omega)$, which is a contradiction.  We can see that for all $\varepsilon\le\varepsilon^*$, $\hat u_{\varepsilon}$ is a local solution of problem $\mathcal P_{\varepsilon}$. Indeed, if $u\in\mathcal U\cap B_{L^1(\Omega)}(\hat u_{\varepsilon}; \alpha^*/2)$, then
	\begin{align*}
		|u-\bar u|_{L^1(\Omega)}\le |u-\hat u_{\varepsilon}|_{L^1(\Omega)}+|\hat u_{\varepsilon}-\bar u|_{L^1(\Omega)}\le \alpha^*,
	\end{align*}
	and consequently, as $\hat u_{\varepsilon}$ is a global solution of problem $\mathcal P^*_{\varepsilon}$,
	\begin{align*}
		\mathcal J(\hat u_{\varepsilon})+\frac{\varepsilon}{2}|\hat u_{\varepsilon}|_{L^2(\Omega)}\le \mathcal J(u)+\frac{\varepsilon}{2}|u|_{L^2(\Omega)}.
	\end{align*}
	The result follows defining $\varepsilon_\alpha:=\varepsilon^*$.

\end{proof}

\begin{thrm}\label{Treg}
	Let Assumption \ref{A2} be fulfilled. Then there exist positive numbers $\alpha,\kappa$ and $\varepsilon_0$ such that for every $\varepsilon\in(0,\varepsilon_0)$ problem $\mathcal P_{\varepsilon}$ has a local solution $\hat u_{\varepsilon}\in\mathbb B_{L^1}(\bar u;\alpha)$. Moreover, 
	\begin{align} \label{Ereg}
	|\hat u_{\varepsilon}-\bar u|_{L^1(\Omega)}\le \kappa\hspace*{0.03cm}\varepsilon^{1/k^*}
	\end{align}
	for every local solution $\hat u_{\varepsilon}$ of problem $\mathcal P_{\varepsilon}$ such that $\varepsilon\in(0,\varepsilon_0)$ and $|\hat u_{\varepsilon}-\bar u|_{L^1(\Omega)} \le  \alpha$.
\end{thrm}   

\begin{proof}
	The first claim follows from Lemma \ref{lemfin}.
	Let $\alpha,\alpha'$and $c$ be the positive numbers in Theorem \ref{Thmreg0}. Define $\eta_\varepsilon:\mathbb R\to\mathbb R$ by $\eta_\varepsilon(u):=\varepsilon u^2/2$ and $\zeta_\varepsilon:=(0,\eta_\varepsilon)\in\Upsilon$ for each $\varepsilon>0$. Note that
	\begin{align*}
	d_C(\eta_\varepsilon,0):=\sum_{m=1}^\infty \frac{1}{2^m}\frac{\varepsilon m^2/2}{1+\varepsilon m^2/2}=\varepsilon\sum_{m=1}^\infty \frac{1}{2^m}\frac{m^2}{2+\varepsilon m^2}\le\varepsilon\sum_{m=1}^\infty\frac{m^2}{2^{m+1}}=3\varepsilon
	\end{align*}
	for all $\varepsilon>0$. Analogously,
	\begin{align*}
	d_C(\frac{\partial \eta_\varepsilon}{\partial u},0):=\sum_{m=1}^\infty \frac{1}{2^m}\frac{\varepsilon m}{1+\varepsilon m}\le\varepsilon\sum_{m=1}^\infty\frac{m}{2^m}=2\varepsilon
	\end{align*}
	for all $\varepsilon>0$. We conclude that $d_{\Upsilon}(\zeta_\varepsilon,0)\le5\varepsilon\le \alpha'$
	for all $\varepsilon\in (0,\varepsilon_0)$, where $\varepsilon_0:=\alpha'/5$. By Theorem \ref{Thmreg0}, 
	\begin{align*}
	|\hat u_\varepsilon-\bar u|_{L^1(\Omega)}\le 5^{\frac{1}{k^*}}c\varepsilon^\frac{1}{k^*}
	\end{align*}
	for all $\varepsilon\in(0,\varepsilon_0)$ such that $|\hat u_\varepsilon-\bar u|_{L^1(\Omega)}\le\alpha$. 
 \end{proof}


\section{Assumptions related {to} subregularity}\label{Section5}

In this section, we gather some results concerning Assumption \ref{A2}, in order to provide sufficient 
conditions under which it is fulfilled. Furthermore, we analyze related assumptions and their relation between themselves. 
Recall that $\bar u\in\mathcal U$ is a local solution of problem (\ref{cost})--(\ref{system}).
Since $\bar u\in\mathcal U$ satisfies the variational inequality (\ref{varin}), we have
\begin{align*}
\bar u(x)=\left\{ \begin{array}{lcc}
b_{1}(x) &   if  & \sigma_{\bar u}(x)>0 \\
\\ b_2(x)&  if & \sigma_{\bar u}(x)<0.
\end{array} \right.
\end{align*}

We introduce the following extended cone suggested in \cite{Casascone}. For a fixed $\tau>0$ define
\begin{align*}
C_{\bar u}^\tau=\left\lbrace v\in L^2(\Omega): v(x)\left\{ \begin{array}{cll}
=0&  \text{if} & |\sigma_{\bar u}(x)|>\tau \mbox{ or } \bar u(x) \in (b_1(x), b_2(x)) \\
\ge 0 &   \text{if}  & |\sigma_{\bar u}(x)| \leq \tau \mbox{ and } \bar u(x) = b_1(x) \\
\le 0&  \text{if} & |\sigma_{\bar u}(x)| \leq \tau \mbox{ and } \bar u(x) = b_2(x)
\end{array} \right. \right\rbrace .
\end{align*}

We introduce the following modification of  Assumption \ref{A2}.

\bino
{\bf Assumption 2${}'$.} {\em
	There exist positive numbers $\alpha_0$ and $\gamma_0$ such that
	\begin{align*}
	\int_{\Omega}\sigma_{\bar u}(u-\bar u)\,dx+\Lambda(u-\bar u)\ge\gamma_0|u-\bar u|_{L^1(\Omega)}^{{k^*}+1},
	\end{align*}
	for all $u\in\mathcal U$ with $u - \bar u \in C_{\bar u}^\tau \cap {\mathbb B}_{L^1(\Omega)}(\bar u; \al_0)$.
}

\bino
This assumption is seemingly weaker than Assumption \ref{A2}. However, we will prove that the two assumptions are equivalent.
Before that, for technical purposes, we introduce the bilinear form $\Gamma:L^{2}(\Omega)\times L^{2}(\Omega)\to\mathbb R$ given by
\begin{align}\label{bilfor}
\Gamma(v_1,v_2):=\frac{1}{2}\int_{\Omega}\Big[\pi_{v_1}v_2+\pi_{v_2}v_1\Big]\,dx.
\end{align}
The bilinear form is particularly useful because of the following property.
\begin{align}\label{remarkgamma}
\Lambda(v_1+v_2)=\Gamma(v_1,v_1)+2\hspace*{0.02cm}\Gamma(v_1,v_2)+\Gamma(v_2,v_2)\quad\forall v_1,v_2\in L^2(\Omega).
\end{align}
We will require the following technical lemma.
\begin{lmm}\label{Lemgamma}
	For every positive number $M$, there exists a positive number $c$ such that 
	\begin{align*}
	|\Gamma(v_1,v_2)|\le c|v_1|_{L^1(\Omega)}^{1/2}|v_2|_{{L^1(\Omega)}}
	\end{align*}
	for all $v_1,v_2\in \mathbb B_{L^\infty}(0;M)$.
\end{lmm}
\begin{proof}
	By Proposition \ref{furprop3}, there exist $c_1,c_2>0$ such that $|\pi_{v}|_{L^\infty(\Omega)}\le c_1|v|_{L^2(\Omega)}$ and $|\pi_{v}|_{L^2(\Omega)}\le c_2|v|_{L^1(\Omega)}$ for all $v\in L^2(\Omega)$. Let $M>0$ be arbitrary. Observe that 
	\begin{align*}
	\Big|\int_{\Omega}\pi_{v_1}v_2\, dx\Big|\le |\pi_{v_1}|_{L^\infty(\Omega)}|v_2|_{L^1(\Omega)}\le c_1M^{\frac{1}{2}}|v_1|_{L^1(\Omega)}^{\frac{1}{2}}|v_2|_{L^1(\Omega)},
	\end{align*}
	and that 
	\begin{align*}
	\Big|\int_{\Omega}\pi_{v_2}v_1\, dx\Big|\le |\pi_{v_2}|_{L^2(\Omega)}|v_1|_{L^2(\Omega)}\le c_2M^{\frac{1}{2}}|v_1|_{L^1(\Omega)}^{\frac{1}{2}}|v_2|_{L^1(\Omega)}
	\end{align*}
	for all $v_1,v_2\in \mathbb B_{L^\infty}(0;M)$. There result follows defining $c:=2^{-1}(c_1+c_2)M^{\frac{1}{2}}$.
\end{proof}

\begin{prpstn} \label{PA2}
	Assumptions \ref{A2} and 2${\,}'$ are equivalent.
\end{prpstn}
\begin{proof}{}
	Clearly Assumption \ref{A2} implies 2${\,}'$. 
	Let $\alpha_0$ and $\gamma_0$ be the numbers in Assumption 2${\,}'$. Let $u\in\mathcal U$ and define
	\begin{align*}
	v_1(x):=\left\{ \begin{array}{lcc}
	u(x)-\bar u(x) &   if  & |\sigma_{\bar u}(x)|\le\tau \\
	\\ 0&  if & |\sigma_{\bar u}(x)|>\tau,
	\end{array} \right.
	\end{align*}
	and
	\begin{align*}
	v_2(x):=\left\{ \begin{array}{lcc}
	0&   if  & |\sigma_{\bar u}(x)|\le\tau \\
	\\ u(x)-\bar u(x)&  if & |\sigma_{\bar u}(x)|>\tau.
	\end{array} \right.
	\end{align*}
	Clearly $v_1\in C_{\bar u}^\tau$ and $v_1+v_2=u-\bar u$.  Let $M$ be a bound for $\mathcal U$ in $L^\infty(\Omega)$, and let $c$ be the positive number in Lemma \ref{Lemgamma} corresponding to $2M$. 
	By Assumption 2${\,}'$,
	\begin{align*}
	\int_{\Omega}\sigma_{\bar u}(u-\bar u)\,dx&=\int_{\Omega}\sigma_{\bar u}v_1\,dx+\int_{|\sigma_{\bar u}|>\tau}\sigma_{\bar u}v_2\,dx\\
	&=\int_{\Omega}\sigma_{\bar u}v_1\,dx+\Lambda(v_1)-\Lambda(v_1)+\int_{|\sigma_{\bar u}|>\tau}\sigma_{\bar u}v_2\,dx\\
	&\ge \gamma_0|v_1|^{k+1}+\tau|v_2|_{L^1(\Omega)}-\Lambda(v_1),
	\end{align*}
	and
	\begin{align*}
	\Lambda(u-\bar u)&=\Lambda(v_1)+2\Gamma(v_1,v_2)+\Lambda(v_2)\\
	&\ge\Lambda(v_1)-2c|v_1|_{L^1(\Omega)}^{1/2}|v_2|_{{L^1(\Omega)}}-c|v_2|_{L^1(\Omega)}^{1/2}|v_2|_{{L^1(\Omega)}}\\
	&\ge\Lambda(v_1)-3c|v_2|_{L^1(\Omega)}|u-\bar u|_{L^1(\Omega)}^{1/2}
	\end{align*}
	for $u\in\mathcal U$ with $|u-\bar u|_{{L^1(\Omega)}}\le\alpha_0$. Thus
	\begin{align*}
	\int_{\Omega}\sigma_{\bar u}(u-\bar u)\,dx+\Lambda(u-\bar u)&\ge\gamma_0|v_1|^{k+1}+\tau|v_2|_{L^1(\Omega)}-3c|v_2|_{L^1(\Omega)}|u-\bar u|_{L^1(\Omega)}^{1/2}\\
	&= \gamma_0|v_1|^{k+1}+|v_2|_{L^1(\Omega)}\Big(\tau-3c|u-\bar u|_{L^1(\Omega)}^{1/2}\Big)
	\end{align*}
	for $u\in\mathcal U$ with $|u-\bar u|_{{L^1(\Omega)}}\le\alpha_0$. Now, by the reverse triangle inequality and Bernoulli's inequality (consider without loss of generality $u\neq\bar u$)
	\begin{align*}
	|v_1|_{L^1(\Omega)}^{k+1}&=|(u-\bar u)-v_2|_{L^1(\Omega)}^{k+1}\ge \Big(|u-\bar u|_{L^1(\Omega)}-|v_2|_{L^1(\Omega)}\Big)^{k+1}\\
	&=|u-\bar u|_{L^1(\Omega)}^{k+1}\Big(1-\frac{|v_2|_{L^1(\Omega)}}{|u-\bar u|_{L^1(\Omega)}}\Big)^{k+1}\ge|u-\bar u|_{L^1(\Omega)}^{k+1}\Big(1-(k+1)\frac{|v_2|_{L^1(\Omega)}}{|u-\bar u|_{L^1(\Omega)}}\Big)\\
	&=|u-\bar u|_{L^1(\Omega)}^{k+1}-(k+1)|u-\bar u|_{L^1(\Omega)}^k|v_2|_{L^1(\Omega)}.
	\end{align*}
	Consequently,
	
	\begin{align*}
	\int_{\Omega}\sigma_{\bar u}(u-\bar u)\,dx+\Lambda(u-\bar u)&\ge\gamma_0|v_1|^{k+1}+|v_2|_{L^1(\Omega)}\Big(\tau-3c|u-\bar u|_{L^1(\Omega)}^{1/2}\Big)\\
	&\ge\gamma_0|u-\bar u|_{L^1(\Omega)}^{k+1}-\gamma_0(k+1)|u-\bar u|_{L^1(\Omega)}^k|v_2|_{L^1(\Omega)}+|v_2|_{L^1(\Omega)}\Big(\tau-3c|u-\bar u|_{L^1(\Omega)}^{1/2}\Big)\\
	&\ge\gamma_0|u-\bar u|_{L^1(\Omega)}^{k+1}+|v_2|_{L^1(\Omega)}\Big(\tau-\gamma_0(k+1)|u-\bar u|_{L^1(\Omega)}^k-3c|u-\bar u|_{L^1(\Omega)}^{1/2}\Big).
	\end{align*}
	Choosing $\alpha$ small enough, one can ensure 
	\begin{align*}
	\int_{\Omega}\sigma_{\bar u}(u-\bar u)\,dx+\Lambda(u-\bar u)&\ge \gamma_0|u-\bar u|_{L^1(\Omega)}^{k+1}+|v_2|_{L^1(\Omega)}\Big(\tau-\gamma_0(k+1)|u-\bar u|_{L^1(\Omega)}^k-3c|u-\bar u|_{L^1(\Omega)}^{1/2}\Big)\\
	&\ge \gamma_0|u-\bar u|_{L^1(\Omega)}^{k+1}+\frac{\tau}{2}|v_2|_{L^1(\Omega)}\ge \gamma_0|u-\bar u|_{L^1(\Omega)}^{k+1}
	\end{align*}
	for all $u\in\mathcal U$ with $|u-\bar u|_{L^1(\Omega)}\le\alpha$.
	
\end{proof}

Proposition \ref{PA2} allows to split Assumption \ref{A2} in two parts, as it follows in the next theorem.

\begin{thrm} \label{TSsR_2}
	Let there exist numbers $\mu_1,\mu_2\in\mathbb R$  and $\alpha>0$ such that 
	\begin{align} \label{EH01}
	\int_{\Omega}\sigma_{\bar u} v \,dx \ge \mu_1 |v|_{L^1(\Omega)}^{{k^*}+1}
	\end{align}
	and 
	\begin{align} \label{EH02}
	\Lambda(v)\ge \mu_2|v|_{L^1(\Omega)}^{{k^*}+1}
	\end{align}
	for every $v  \in (\mathcal U - \bar u) \cap C_{\bar u}^\tau \cap {\mathbb B}_{L^1(\Omega)}(\bar u; \al)$. 
	If $\mu_1 + \mu_2 > 0$, then Assumption 2 is fulfilled, hence the optimality mapping $\Phi$
	(see (\ref{optmapping})) of problem (\ref{cost})--(\ref{system}) is strongly H\"older subregular with exponent 
	$\lambda = 1/k^*$ at the reference point $(\bar y, \bar p, \bar u)$.
\end{thrm}

The proof consists of summation of (\ref{EH01}) and (\ref{EH02}) and utilization of Proposition \ref{PA2} and 
Theorem \ref{Ssr}.

\bino
The splitting of Assumption 2 has the advantage that the inequalities in  (\ref{EH01}) and (\ref{EH02}) can be 
analyzed separately. The next proposition is related to (\ref{EH01}).

The following  assumption has become standard in the literature on PDE optimal control problems with 
bang-bang controls, see, e.g., \cite{Casasbang,Hinzestruct,Wachelliptic,Wachstruct}.  

\begin{ssmptn}\label{A3}
	There exists a  positive number $\mu_0$ such that 
	\begin{align*}
	\text{meas}\left\lbrace x\in\Omega: |\sigma_{\bar u}(x)|\le\varepsilon\right\rbrace
	\le \mu_0\varepsilon^{\frac{1}{k^*}}\quad\forall \varepsilon>0.
	\end{align*}	
\end{ssmptn}

\begin{prpstn} \label{PA3_1}
	The following statements hold.
	\begin{itemize}
		\item[(i)]  If Assumption \ref{A3} is fulfilled then there exists $\mu_1 > 0$
		such that (\ref{EH01}) holds for every $v \in \mathcal U -\bar u$. 
		\item[(ii)]  Suppose there exists $\nu > 0$ such that $b_2(x) - b_1(x) \geq \nu$ for a.e. $x \in \Omega$.
		If (\ref{EH01}) holds for every $v \in \mathcal U -\bar u$ then Assumption \ref{A3} is fulfilled.
	\end{itemize}
\end{prpstn}

\begin{proof}
	The proof of the first claim follows \cite[Proposition 3.1]{Wachelliptic}, see also \cite[Proposition 2.7]{Casasbang}. 
	It has been also proved several times in the literature on ordinary differential equations 
	in a somewhat stronger form; see, e.g., \cite{Sey2,SubregOsm,Prei,Sey1}.
	
	Let us prove the second claim.
	For each $\varepsilon>0$, define
	\begin{align*}
	u_\varepsilon(x):=\left\{ \begin{array}{lcc}
	\bar u(x) &   \text{if}  & |\sigma_{\bar u}(x)|>\varepsilon \\
	\\ b_1(x)&  \text{if} & |\sigma_{\bar u}(x)|\le\varepsilon\quad\text{and}\quad
	\bar u(x)\in\Big[\displaystyle\frac{b_1(x)+b_2(x)}{2},b_2(x)\Big]\\
	\\ b_2(x)&  \text{if} & |\sigma_{\bar u}(x)|\le\varepsilon\quad\text{and}\quad 
	\bar u(x)\in\Big[b_1(x),\displaystyle\frac{b_1(x)+b_2(x)}{2}\Big).
	\end{array} \right.
	\end{align*}
	Clearly each $u_\varepsilon$ belongs to $\mathcal U$, and 
	\begin{align}\label{onehalfeine}
	|u_\varepsilon(x)-\bar u(x)|\ge\frac{1}{2}|b_2(x)-b_1(x)|
	\end{align}
	for a.e $x\in\left\lbrace s\in\Omega: |\sigma_{\bar u}(s)|\le\epsilon\right\rbrace $. From  (\ref{EH01}) we have
	\begin{align*}
	\mu_1\Big(\int_{|\sigma_{\bar u}|\le\varepsilon} |u_{\varepsilon}-\bar u|\,dx\Big)^{k+1}
	\le\int_{|\sigma_{\bar u}|\le\varepsilon} \sigma_{\bar u}(u_{\varepsilon}-\bar u)\,dx
	\le\varepsilon\int_{|\sigma_{\bar u}|\le\varepsilon} |u_{\varepsilon}-\bar u|\,dx.
	\end{align*}
	This implies
	\begin{align}\label{wsfas}
	\int_{|\sigma_{\bar u}|\le\varepsilon} |u_{\varepsilon}-\bar u|\,dx\le\mu_1^{-\frac{1}{k}}\varepsilon^\frac{1}{k}.
	\end{align}
	Using (\ref{onehalfeine}) and (\ref{wsfas}) we obtain that
	\begin{align*}
	\text{meas} \left\lbrace x \in \Omega: |\sigma_{\bar u}(x)| \,\le \, \varepsilon \right\rbrace&
	=\frac{1}{\nu} \int_{|\sigma_{\bar u}|\le\varepsilon} \nu \,dx
	\,\le\,\frac{1}{\nu} \int_{|\sigma_{\bar u}|\le\varepsilon} |b_2-b_1|\,dx
	\, \le \, \frac{2}{\nu} \int_{|\sigma_{\bar u}| \le \varepsilon} |u_{\varepsilon}-\bar u| \,dx\\
	&\le 2(\mu_1)^{-\frac{1}{k}} \nu^{-1}  \, \varepsilon^\frac{1}{k}.
	\end{align*}
	Thus Assumption \ref{A3} is fulfilled with $\mu_0 :=2 (\mu_1)^{-\frac{1}{k}} \nu^{-1}$.
\end{proof}

\bibliography{references}{}
\bibliographystyle{abbrv}
\end{document}